\documentclass[reqno]{amsart}
\usepackage{amscd}
\usepackage[dvips]{graphics}

\def\beq{\begin{equation}}
\def\eeq{\end{equation}}
\def\ba{\begin{array}}
\def\ea{\end{array}}

\newenvironment{abs}{\textbf{Abstract}\mbox{  }}{ }
\newenvironment{key words}{\textbf{Keywords}\mbox{  }}{ }
\newtheorem{thm}{Theorem}[section]
\newtheorem{lm}[thm]{Lemma}
\newtheorem{prop}[thm]{Proposition}
\newtheorem{cor}[thm]{Corollary}

\theoremstyle{definition}
\renewenvironment{proof}{\noindent{\textbf{Proof.}}}{\hfill$\Box$}
\newtheorem{rem}[thm]{Remark}

\newtheorem{df}[thm]{Definition}

\theoremstyle{remark}
 
\numberwithin{equation}{section}
\begin{document}
%old title: On an extension of sharp Hardy-Littlewood-Sobolev inequality and related Sharp Sobolev inequalities on sphere 4-26-2012

\pagestyle{plain}
%\today completed on 4-29-2012. MJ
\title{Reversed Hardy-Littewood-Sobolev inequality}
%On the extension of sharp Hardy-Littewood-Sobolev inequality II: the negative exponent}
\author  {Jingbo Dou and Meijun Zhu}
\address{Jingbo Dou, School of Statistics, Xi'an University of Finance and
Economics, Xi'an, Shaanxi, 710100, China and Department of Mathematics, The University of Oklahoma, Norman, OK 73019, USA}

\email{jbdou@xaufe.edu.cn}

\address{ Meijun Zhu, Department of Mathematics,
The University of Oklahoma, Norman, OK 73019, USA}

\email{mzhu@math.ou.edu}

\maketitle

\begin{abs}
In this paper, we obtain a reversed  Hardy-Littlewood-Sobolev inequality:  for $0<p, \, t<1$ and $\lambda=n-\alpha <0$ with $
1/p +1 /t+ \lambda /n=2$, there is a best constant $N(n,\lambda,p)>0$, such that
\begin{equation*}\label{1-1}
|\int_{\mathbb{R}^n} \int_{\mathbb{R}^n} f(x)|x-y|^{-\lambda} g(y) dx dy|\ge N( n,\lambda,p)||f||_{L^p(\mathbb{R}^n)}||g||_{L^t(\mathbb{R}^n)}
\end{equation*}
holds for all nonnegative  functions  $f\in L^p(\mathbb{R}^n), \, g\in L^t(\mathbb{R}^n).$
For $p=t$, we prove the existence of extremal
functions,  classify all extremal functions via the method of moving sphere,  and compute the best constant.
\end{abs}

\smallskip

%---------------------------------------------------------------------------------

\section{Introduction\label{Section 1}}
The classic sharp Hardy-Littlewood-Sobolev (HLS) inequality (\cite{HL1928, HL1930, So1963, Lieb1983}) states that
\begin{equation}\label{class-HLS}
|\int_{\mathbb{R}^n} \int_{\mathbb{R}^n} f(x)|x-y|^{-(n-\alpha)} g(y) dx dy|\le N(n,\lambda,p)||f||_{L^p(\mathbb{R}^n)}||g||_{L^t(\mathbb{R}^n)}
\end{equation}
holds for all $f\in L^p(\mathbb{R}^n), \, g\in L^t(\mathbb{R}^n),$ $1<p, \, t<\infty$, $0<\lambda:=n-\alpha <n$ and
$
1/p +1 /t+ \lambda/n=2.$
Lieb \cite{Lieb1983} proved the existence of the extremal functions to the inequality with sharp constant and computed the best constant in the case of $p=t$ (or one of these two parameters is two). The sharp HLS inequality implies sharp Sobolev inequality, Moser-Trudinger-Onofri and Beckner inequalities \cite{B1993}, as well as Gross's logarithmic Sobolev inequality \cite{Gr}. All these inequalities play significant role in solving global geometric problems, such as Yamabe problem, Ricci flow problem, etc. Besides  recent extension of the sharp HLS on the Heisenberg group by Frank and Lieb \cite{FL2012},
there are at least two other directions concerning the extension of the above sharp HLS inequality: (1) Extending the sharp inequality on general manifolds, see, for example, Dou and Zhu \cite{DZ2012-2} for such an extension on the upper half space and related research; (2) Extending it for the negative exponent $\lambda$ (that is for the case of $\alpha>n$). In this paper, we extend the sharp HLS inequality for the negative exponent $\lambda$.

More specifically, in this paper, we prove that the reversed
Hardy-Littlewood-Sobolev inequality  for $0<p, \, t<1$, $\lambda<0$ holds for all nonnegative $f\in L^p(\mathbb{R}^n), \, g\in L^t(\mathbb{R}^n).$ For $p=t$, the existence of extremal
functions is proved, all extremal functions are classified via the method of moving sphere,  and the best constant is computed.

Prior to our research, it seems that the only result concerning $\lambda<0$ was discussed by Stein and Weiss \cite{SW1960} in 1960, where they showed that a HLS inequality (not in the sharp form) for $p\in ((n-1)/n, n/\alpha)$  holds (Theorem G in \cite{SW1960}). However,  the range for $p$ does not include the important conformal invariant case $p= t=2n/(n+\alpha)$, thus it seems hard to find the sharp constant. On the other hand, recent results on sharp Sobolev type inequalities with negative exponents on $\mathbb{S}^n$, (see, e.g. Yang and Zhu \cite{YZ2004}, Hang and Yang \cite{HY2004} for the Paneitz operator on $\mathbb{S}^3$,  and Ni and Zhu \cite{NZ2008a} for the Laplacian operator on $\mathbb{S}^1$),  strongly indicate
  that certain HLS inequalities for $\lambda<0$ shall hold.
  %, we ask: what is the analogous HLS if $\lambda<0$ and both $p, t \in (0, 1)$?

%  For instance, it was shown in \cite{NZ2008a} that \eqref{1-1} with a reverse inequality hold on $\mathbb{S}^1$ with $p=-2$. We may ask: is it the case for all negative $p$? One may ask a more broad question: is there any similar inequality to \eqref{1-1} for $\lambda<0$?

The main purpose of this paper is to establish the following reversed  HLS inequality and its sharp form.

% we give a complete answer to the above question.

%Our main result is the following theorem.

\begin{thm}\label{ext-HLS}For $n\ge1, 0<p, \, t<1$ and $\lambda=n-\alpha<0$ satisfying
\begin{equation}\label{1-2}
\frac1 p +\frac 1 t+\frac \lambda n=2,
\end{equation}
there is a best constant $N^{*}( n,\alpha, p)>0$, such that, for all nonnegative $f\in L^p(\mathbb{R}^n), \, g\in L^t(\mathbb{R}^n),$
\begin{equation}\label{HLS-ineq-1}
\int_{\mathbb{R}^n} \int_{\mathbb{R}^n} f(x)|x-y|^{-\lambda} g(y) dx dy\ge N^*(n,\alpha, p)||f||_{L^p(\mathbb{R}^n)}||g||_{L^t(\mathbb{R}^n)}.
\end{equation}
 \end{thm}

 For $p<1$, the convention notation for $f(x) \in L^p(\mathbb{R}^n)$  means $\int_{\mathbb{R}^n}|f(x)|^p dx <\infty$.

 %Since inequality \eqref{HLS-ineq-1} is reversed (comparing with \eqref{class-HLS}), the usual density argument can not be used here. Namely, we have to prove the above inequality for all functions  $f\in L^p(\mathbb{R}^n)$ and $g \in L^t(\mathbb{R}^n)$. This is also the major difficulty for us to derive the sharp form for the inequality on $\mathbb{R}^n.$
%

%%revised 4-15-2013

 For $p=t=2n/(n+\alpha)$, we are able to compute the sharp constant. In this case, inequality \eqref{HLS-ineq-1} is equivalent to the following reversed HLS on sphere $\mathbb{S}^n$: for all nonnegative $F\in L^p(\mathbb{S}^n), \, G\in L^p(\mathbb{S}^n),$
\begin{equation}\label{HLS-ineq-1-sph}
\int_{\mathbb{S}^n} \int_{\mathbb{S}^n} F(\xi)|\xi-\eta|^{\alpha-n} G(\eta) dS_\xi dS_\eta\ge N^*(n,\alpha)||F||_{L^p(\mathbb{S}^n)}||G||_{L^p(\mathbb{S}^n)},
\end{equation}
 where and throughout the paper $|\xi-\eta|$ is denoted as the chordal distance from $\xi$ to $\eta$ in $\mathbb{R}^{n+1}$, and $N^*(n,\alpha)$ is the same as $N^*(n, \alpha, 2n/(n+\alpha))$.

For $\alpha \in (0,\infty),$ define the classic singular integral operator on $\mathbb{S}^n$  by
\begin{equation} \label{sig-ope-1}
\tilde I_\alpha F(\xi)=\int_{\mathbb{S}^n}\frac{F(\eta)}{|\xi-\eta|^{n-\alpha}}dS_\eta, \quad ~~~~~~~~~~\forall \xi\in\mathbb{S}^n.
\end{equation}
We have

%Then \eqref{HLS-ineq-1-sph} is equivalent to
%\begin{equation}\label{HLS-ineq-1-sph-1}
% ||\tilde I_\alpha F||_{L^{\frac{2n}{n-\alpha}}} \ge  N^*(n,\alpha)||F||_{L^{\frac{2n}{n+\alpha}}}
%\end{equation}
%for all nonnegative $F(\xi)\in L^{\frac{2n}{n+\alpha}}(\mathbb{S}^n).$

 %It turns out, for $p=2n/(n+\alpha)$, under a mild condition, we are able to establish a sharp reversed HLS inequality on $\mathbb{S}^n,$ and to compute the best constant.

%Our sharp reversed HLS inequality on sphere can be stated as

 \begin{thm}\label{ext-HLS-exis} Let  $1\le n <\alpha$.  For all nonnegative $F\in L^{2n/(n+\alpha)}(\mathbb{S}^n)$,
\begin{equation}\label{HLS-ineq-2}
||\tilde I_\alpha F||_{L^{\frac{2n}{n-\alpha}}(\mathbb{S}^n)} \ge N^{*}(n, \alpha)||F||_{L^{\frac{2n}{n+\alpha}}(\mathbb{S}^n)},
\end{equation}
where the best constant
\begin{equation}\label{extreCon}
N^{*}(n,\alpha)=\pi^{(n-\alpha) /2} \frac{\Gamma(\alpha/2)}{\Gamma(n/2+\alpha/2)} \{\frac{\Gamma(n/2)}{\Gamma(n)} \}^{-\alpha/n};
\end{equation}
And equality holds if and only if
\begin{equation*}\label{extre}
F(\xi)=a(1-\xi\cdot \eta)^{-\frac{n+\alpha} 2}
\end{equation*}
for some $a>0$ and $\eta\in \mathbb{R}^{n+1}$ with $|\eta|<1$.
\end{thm}

%Note that $L^{2n/(n+\alpha)}(\mathbb{S}^n) \subset L^1(\mathbb{S}^n)$, thus $N_1^{*}(n,\alpha) \ge N^{*}(n,\alpha).$ It is not clear  whether \eqref{HLS-ineq-2} holds for all $F(\xi)\in L^{2n/(n+\alpha)}(\mathbb{S}^n)$ or not. Namely, it is an open question to verify or disprove that $N_1^{*}(n,\alpha) = N^{*}(n,\alpha).$ {\bf question is solved after the AMS-Colorado meeting. See C. Li there.

For $\alpha \in (0,\infty),$ define the classic singular integral operator on $\mathbb{R}^n$  by
\begin{equation} \label{sig-ope-3}
I_\alpha f(x)=\int_{\mathbb{R}^n}\frac{f(y)}{|x-y|^{n-\alpha}}dy, \quad ~~~~~~~~~~\forall x\in\mathbb{R}^n.
\end{equation}
From Theorem \ref{ext-HLS-exis} and a stereographic projection, we have the sharp reversed HLS inequality on $\mathbb{R}^n$ for $p=t=2n/(n+\alpha)$.

\begin{cor}\label{cor3.1}Let $1\le n <\alpha.$  For all nonnegative function $f\in L^{2n/(n+\alpha)}(\mathbb{R}^n)$,
\begin{equation}\label{HLS-ineq-2-mj-8-5}
||I_\alpha f||_{L^{\frac{2n}{n-\alpha}}(\mathbb{R}^n)} \ge N^{*}(n, \alpha)||f||_{L^{\frac{2n}{n+\alpha}}(\mathbb{R}^n)},
\end{equation}
where $N^*(n, \alpha)$ is given by \eqref{extreCon};
And the equality holds if and only if
\begin{equation*}
 f(x)=c\big(\frac{1}{|x-x_0|^2+d^2}\big)^\frac{n+\alpha}2
\end{equation*}
for some $c,\ d>0,$ and $x_0\in \mathbb{R}^n.$
\end{cor}

%The inequality directly follows from Theorem \ref{ext-HLS-exis} and a stereographic projection. However, it is an open question: does \eqref{HLS-ineq-2-mj-8-5} hold for all $f\in L^{2n/(n+\alpha)}(\mathbb{R}^n)$?
%
%
%As a corollary, we have a sharp reversed HLS on $\mathbb{R}^n$ for all continuous functions $f(x)\in C(\mathbb{R}^n)$ with certain decay rate at infinity, see Corollary \ref{cor3.1}. Again, the sharp inequality for a large set (say, e.g. $f(x)\in L^{2n/(n+\alpha)}(\mathbb{R}^n)$) is unknown.

%{\bf sharp inequality on $R^n$ as a corollary, at least for $f\in L^p$ and $suppf \subset B_R(0)$ for some large $R$. maybe point out the relation to convex geometry, the chord integral? 7-27-2012}

%In particular, we do not know whether $N^{*}(n, \alpha)$ is strictly bigger than $N^{**}(n, \alpha, 2n/(n+\alpha))$ or not.

%For $p= 2n/(n+\alpha)$, inequality on $S^n$ is much clean and sharp later {\bf 7-18-2012}

%Following Beckner's approach \cite{B1993}, we can derive the  following sharp Sobolev inequalities on $\mathbb{S}^n$.
%
%{\bf will modify the follows later 3-27-2012}
%\begin{thm}\label{Sobolev} Let $n\ge1, p<0$. If $d \xi$ is the normalized surface measure of the standard sphere $\mathbb{S}^n$ such that $\int_{\mathbb{S}^n} d \xi=1$, then
%
%(1) for all $f(\xi)\in W^{1,p}(\mathbb{S}^n)$
%\begin{equation}\label{1-5}
%(\int_{\mathbb{S}^n}|f(\xi)|^p d\xi)^{2/p}\ge \frac{p-2}n  \int_{\mathbb{S}^n}|\nabla f(\xi)|^2 d\xi +\int_{\mathbb{S}^n}|f(\xi)|^2 d\xi.
%\end{equation}

\smallskip

We outline the strategy in proving above theorems.
 The proof of Theorem \ref{ext-HLS}  is along the line of
the proof for the classic HLS inequality (see, e.g. Stein \cite{Stein}). The main difference is that our inequality is reversed. The reversed H\"older inequality, converse Young's inequality, as well as a new established  Marcinkiewicz interpolation involving exponents less than $1$ (could be negative) are used. Our proof for the existence
of extremal functions is quite different to that for the sharp HLS
inequality (due to Lieb \cite{Lieb1983}), and we can only obtain the result for $p=t=2n/(n+\alpha).$  We first prove Theorem \ref{ext-HLS-exis}
for $F(\xi)\in L^1(\mathbb{S}^n$).  A density lemma (Lemma \ref{lem3.1})
will be established, which allows us to reduce the proof for all $L^1$ functions to
 continuous functions on $\mathbb{S}^n.$ The extra condition for
 functions (i.e. $F(\xi)\in L^1(\mathbb{S}^n$)) will be removed while considering
 its dual form (inequality \eqref{HLS-ineq-1-sph})\footnote{We thank C. Li, whose comment leads to
 the removal of the extra condition.}.
%See more details in Remark \ref{rem_lemma3.1}. Removed on 4-16-2013
In proving the existence of extremal functions, symmetrization argument is used.  We point out here that for $\alpha>n$, there is a new phenomenon  in proving the convergence of the minimizing sequence $\{F_i\}_{i=1}^\infty$: even $F_i$ has a concentration mass, the mass of $ \tilde I_\alpha F_i$  may not. In other words, the classic concentration compactness argument does not work. In fact, we  show in Remark \ref{pointwise} that even a minimizing sequence $\{F_i\}_{i=1}^\infty$ pointwise converges to $F$,  $\tilde I_\alpha F_j(\xi)$ may not converge to $\tilde I_\alpha F(\xi) $ pointwise.
It is one of the main difficulties to show that there is a subsequence of $\tilde I_\alpha F_j(\xi)$
that is a Cauchy sequence under certain metric.
%%{A crucial Lemma (Lemma ???) similar to Lemma 2.4 in Lieb \cite{Lieb1983} is needed
%so that we can rescale the minimizing sequence to make.
Another difficulty is to classify all extremal functions in order to compute the sharp constant. This is settled via the
method of moving sphere, introduced in Li and Zhu \cite{LZ1995}. Our research certainly answers one of Y.Y. Li's open questions in \cite{Li2004}, where he asks for the background for the study of the integral equation with negative exponents.

Quite natural question after we establish the reversed Hardy-Littlewood-Sobolev inequality is: Can we derive certain Sobolev type inequalities (such as those Sobolev inequalities with negative powers on $\mathbb{S}^1$ and on $\mathbb{S}^3$), and use these Sobolev inequalities to investigate curvature equations (for example, the prescribing $Q-$ curvature on $\mathbb{S}^3$)? It is not obvious that one can derive Sobolev inequalities from the reversed HLS inequality as in the case for HLS inequality. However, we are able to use the reversed HLS inequality directly to derive the existence of solutions to certain curvature equations, see Zhu \cite{Zhu14}. From the view point given in Zhu \cite{Zhu14}, it seems more natural to extend Lieb's  sharp HLS inequality on $\mathbb{S}^n$  to the ones on  general compact Riemannian manifolds, and use them to investigate curvature equations (including a generalized Yamabe problem formulated in \cite{Zhu14}. More details will be given in a forthcoming paper \cite{HZ2014}.

The paper is organized as follows. Theorem \ref{ext-HLS} is proved in Section 2, where a new  Marcinkiewicz interpolation theorem is also stated and proved; Theorem \ref{ext-HLS-exis} is proved in Section 3, where a Liouville theorem (Theorem \ref{sys-critical}) concerning an integral system is also proved.

%And the sharp constant is computed in Section 3. From the argument based on moving sphere method, it is clear that the condition $p=2n/(n+\alpha)$ for classifying extremal functions is related to the conformal invariant property of the integral operator.

%It remains as an open problem to find the best constant for $p\ne 2n/(n+\alpha)$, even for the classic HLS.

\section{reversed Hardy-Littlewood-Sobolev inequality  \label{Section 2}}
In this section, we prove Theorem \ref{ext-HLS}: the reversed HLS inequality (with a rough constant) in $\mathbb{R}^n $ for $\lambda=n-\alpha<0.$

\subsection{\textbf{Some basic inequalities}}\label{subsection 2.1}
%\noindent{\it Notation}:
For $p<1$ and $p\ne 0$, if $f(x)$ satisfies $\int_{\mathbb{R}^n} |f|^p dx<\infty,$ we say $f(x) \in L^p(\mathbb{R}^n),$ and call $(\int_{\mathbb{R}^n} |f|^p)^{1/p} dx $ (denoted as $||f||_{L^p}$ later)  the $L^p$ norm of $f(x)$. The $L^p$ norm for $p<1$  is not a norm for a vector space. Nevertheless, certain integral inequalities still hold.

%It needs to point out that $L^p(\mathbb{R}^n)$ for $p<1$ may not be a vector space and

\medskip

 \noindent{\bf Lemma 2.1 (Reversed H\"{o}lder inequality)}. For $p\in (0, 1)$, $p'=p/(p-1)$,  and nonnegative functions $f\in L^p(\mathbb{R}^n)$ and $g\in L^{p'}(\mathbb{R}^n),$
\[\int_{\mathbb{R}^n}f(x)g(x)dx\ge\|f\|_{L^p }\|g\|_{L^{p'} }.
\]
%Note that $p'<0 $ in  the case  $0<p<1$, so we assume $|g|>0$ a.e..

The reversed H\"{o}lder inequality can be derived easily from the standard H\"{o}lder inequality.

\smallskip

\noindent{\bf Lemma 2.2 (Converse Young's inequality)}. Suppose that  $0<p<1,$ and $q, \, r<0$ are three parameters satisfying $\frac1p+\frac1q=1+\frac1r.$
For any  nonnegative measurable  functions $h,g$, define
% $h\in L^p(\partial\mathbb{R}^n),g\in L^q(\mathbb{R}^n),$
\[ g*h(x)=\int_{\mathbb{R}^n}g(x-y)h(y) dy.
 \]Then
 \begin{equation*}\label{Young-1}
 \|g*h\|_{L^r}\ge \|g\|_{L^q }\|h\|_{L^p }.
 \end{equation*}

 \smallskip

 The proof of the above converse Young's inequality can be found, e.g. in Brascamp and Lieb \cite{BL1976}, where they also identified the best constant for the classic Young's inequality.

 %It is interesting and important to point out that Brascamp and Lieb's inequality is
 % one of the closest inequalities to inequality \eqref{HLS-ineq-1}. keep for myself 4-29-2012

\smallskip

\noindent{\bf Lemma 2.3 (Reversed Minkowski inequality)}. %\label{Mink ineq}
 If $q<0$, then for any  nonnegative measurable  functions $F(x,y)$,
 \begin{equation*}\label{Mink-1}
 \big[\int_Y\big(\int_X F(x,y)d\mu(x)\big)^qd\nu(y)\big]^\frac1q\ge\int_X\big(\int_Y [F(x,y)]^qd\nu(y)\big)^\frac1qd\mu(x)
 \end{equation*}
%\end{lm}

The proof for the reversed Minkowski inequality can be found in \cite{HLP1954} (on $P_{148}$).

\medskip

To establish the reversed HLS inequality, we also need to extend the classic Marcinkiewicz interpolation theorem for $L^p$ function with $p<1$.

% In this section, we present some essential preliminary knowledge and extend the  Marcinkiewicz interpolation theorem with $0<p<1.$
%
% {\bf meaning of $L^p$ for $p<1$, reverse Holder, converse Young inequality}
%
%We firstly recall the reversed :
%Let $X$ be a set of $\mathbb{R}^n $ and
%
%Let $X$ be a set of $\mathbb{R}^n $.

Recall: for a given measurable function $f(x)$ on $\mathbb{R}^n$ and $0<
p<\infty$, the weak $L^p$ norm of $f(x)$  is defined by
\[\|f\|_{L^{p}_W}=\inf \{A>0 \,:\, meas \{ |f(x)|>t\}\cdot t^p\le A^p\},
\]
% and $L^{p}_W=\{f\,|\, f ~\text{is a measurable function, and}~\|f\|_{L^{p}_W}<\infty \}$.
For $p<0$, we define
the weak $L^{p}$ norm for $f(x)$ in a similar way:
$$||f||_{L^{p}_W}:= \sup \{ A>0  : \, m\{|f(x)|<t\} \cdot t^{p} \le A^{p}\}.
$$
Thus, for $p<0$,
$$||f||_{L^{p}_W}^{p}:= \inf \{ B>0  : \, m\{|f(x)|<t\} \cdot t^{p} \le B\}.
$$

\smallskip

%It is well known the Minkowski inequality for $0<p<1$, (see,e.g. \cite{HLP1954})
%\begin{equation*}
%(\int_{\mathbb{R}^n}(f(x)+g(x))^p dx )^\frac 1p \geq (\int_{\mathbb{R}^n} |f(x)|^p dx )^\frac 1p
%+(\int_{\mathbb{R}^n} |g(x)|^p dx )^\frac 1p,
%\end{equation*}
%\begin{equation*}
%(\int_A(\int_B h(x,y)dy )^p dx)^\frac 1p \geq \int_B(\int_A
%|h(x,y)|^p dx)^\frac 1p dx.
%\end{equation*}

Let $T\,:\,L^p(\mathbb{R}^n)\to L^q(\mathbb{R}^n)$ be a linear operator.
We recall that for $0< p, \, q<\infty$, operator $T$ is called the weak type $(p,q)$ if  there exists a constant $C(p,q)>0$ such that for all $f\in L^p(\mathbb{R}^n)$
\begin{equation*}\label{weak-1}
meas\{x:|Tf(x)|>\tau\}\le\big(C(p,q)\frac{\|f\|_{L^p}}\tau\big)^q, ~~~~~~ ~\forall~ \tau>0.
\end{equation*}
Similarly, we can extend the definition of the weak type $(p,q)$ to the case $q<0<p<1$.

\begin{df}\label{pqtype}
For $q<0<p<1$, we say operator $T$ is of the weak type $(p,q)$, if  there exists a constant $C(p,q)>0,$ such that for all $f\in L^p(\mathbb{R}^n),$
\begin{equation*}\label{weak-2}
meas\{x:|T f(x)|<\tau\}\le\big(C(p,q)\frac{\|f\|_{L^p}}\tau\big)^q, ~~~~~~~~ ~\forall~ \tau>0.
\end{equation*}
\end{df}

We now can state the extension to the classic Marcinkiewicz interpolation theorem.

%{\bf I am here 7-26-2012}

\begin{prop}\label{Mar-1}
Let $T$ %\,:\,(X,\mu)\to(Y,\nu)$
be a linear operator which maps any nonnegative function to a nonnegative function. For a pair of  numbers $(p_1,q_1),( p_2,q_2)$ satisfying $q_i<0<p_i<1,\,i= 1,2,$  $p_1< p_2$
and $q_1< q_2$,
if T is weak type $(p_1,q_1) $ and ~$(p_2,q_2) $ for all nonnegative functions, then for any  $\theta\in (0, 1)$,  and
\begin{eqnarray}\label{cond-pq}
\frac 1p&=&\frac{1-\theta}{p_1}+\frac \theta {p_2},\quad\quad
\frac 1q=\frac{1-\theta}{q_1}+\frac \theta q_2,
\end{eqnarray}
then $T$ is reversed strong type ~$(p,q)$ for all nonnegative functions, that is,
 \begin{equation}\label{Mar-ineq}
 \|Tf\|_{L^q}\geq C\|f\|_{L^p}, \quad \forall f\in {L^p}(\mathbb{R}^n) \quad \mbox{and} \quad f\ge 0,
\end{equation}
for some constant $C=C(p_1,p_2,q_1,q_2,\theta)>0.$
 \end{prop}

\noindent \textbf{Proof.} For any measurable function $f(x)$, denote
$\tilde m_g(\tau)= meas \{x \,:\, |g(x)|<\tau\}.$
%throughout this section.
%$m_g(\lambda)= meas \{x \,:\, |g(x)|>\lambda\}, ~~~~~~~~~~~~ $

%To prove this theorem, we need the following lemmas. The following
%\begin{eqnarray*}
%m_f(\lambda)= meas \{x\,:\,| f(x)|>\lambda\},\\
%m_{T f}(\lambda)= meas \{x\,:\,|Tf(x)|<\lambda\}.
%\end{eqnarray*}
Easy to check that
%for $r> 0$, if $g\in L^r(\mathbb{R}^n)$, then
%\begin{equation}\label{m-1}
%m_g(\lambda)\leq\frac{\|g\|^r_{L^r}}{\lambda^r},\quad \forall~ \lambda>0,
%\end{equation}
%and
%\begin{equation}\label{m-2}
%\|g\|^r_{L^r(X)}=r\int^\infty_0 t^{r-1} m_g(t)dt;
%\end{equation}
%And
for $r< 0$, if $g\in L^r(\mathbb{R}^n)$, then
\begin{equation*}\label{m-3}
\tilde m_g(\tau)\leq\frac{\|g\|^r_{L^r}}{\tau^r},\quad \forall~ \tau>0,
\end{equation*}
and
\begin{equation}\label{m-4}
\|g\|^r_{L^r(X)}=|r|\int^\infty_0 t^{r-1} \tilde m_g(t)dt.
\end{equation}

For a nonnegative $f(x)\in L^{p_1}(\mathbb{R}^n)\cap L^{p_2}(\mathbb{R}^n)$ and $\gamma>0$, write $f=f_1+f_2$, where
\begin{equation*}
f_1(x)=
\begin{cases}f(x),\quad&  \text{if}~ f(x)\leq\gamma,\\
0, \quad&\text{if}~ f(x)>\gamma,
\end{cases}
\end{equation*}
and
\begin{equation*}
f_2(x)=\begin{cases}0,\quad& \text{if}~f(x)\leq\gamma,\\
f(x),\quad&\text{if}~f(x)>\gamma.
\end{cases}
\end{equation*}
Thus
\begin{equation*}
\{x\,:\, Tf<\tau\}\subset \{x\,:\, Tf_1<\frac{\tau}{2}\}\cup \{x: Tf_2<\frac{\tau}{2}\}.
\end{equation*}

%Without loss of generality, we only consider the case $q_2<q_1$. The case $q_2>q_1$ can be derived in the same way (by switching $f_1$ with $f_2$ in the following derivation).
%(mj, add on 8-05-2012}
%(MJ 6-14-2014: Assume $q_1<q_2$!

Since $T$ is weak type $(p_1,q_1)$ and $(p_2,q_2),$  we have
\begin{equation*}
\tilde m_{Tf_1}(\tau)\leq c_1(\frac{\|f_1\|_{L^{p_1}}}{\tau})^{q_1}, \  \ \ \
~~~~~~~~~ \mbox{and} ~~~~~~~~~~~~~ \ \ \ \
\tilde m_{Tf_2}(\tau)\leq c_2(\frac{\|f_2\|_{L^{p_2}}}{\tau})^{q_2}.
\end{equation*}
%Without loss of generality, we only consider the case $q_2<q_1$. {\bf The case $q_2>q_1$ can be derived in the same way. will check. I am here now 7-27-2012}
It then follows from  \eqref{m-4}  that
\begin{eqnarray}\label{m-7}
\|Tf\|^q_{L^q}&=&|q|\int_0^\infty t^{q-1} \tilde m_{Tf}(t) dt\nonumber\\
&\leq& |q|\int_0^\infty t^{q-1}[\tilde m_{Tf_1}\{{t/2}\}+ \tilde m_{Tf_2}\{{t/2}\}]dt\nonumber\\
&\leq& C|q|\int_0^\infty t^{q-q_1-1}\|f_1\|^{q_1}_{L^{p_1}}dt+C|q|\int_0^\infty t^{q-q_2-1}\|f_2\|^{q_2}_{L^{p_2}}dt \nonumber\\
%&\le& |q|p_1c_1\int_0^\infty t^{q-q_1-1}\big(\int_0^{\gamma}s^{p_1-1}m_f(s) ds\big)^\frac{q_1}{p_1} dt\nonumber\\
%& &+|q|p_2c_2\int_0^\infty t^{q-q_2-1}\big(\int_{\gamma}^\infty s^{p_2-1}m_f(s) ds\big)^\frac{q_2}{p_2}dt\nonumber\\
&=&C|q|V_1+C|q|V_2,
\end{eqnarray}
where
$$
V_1=   \int_0^\infty t^{q-q_1-1}\|f_1\|^{q_1}_{L^{p_1}}dt, \ \ \ ~~~~~~~~~ V_2 = \int_0^\infty t^{q-q_2-1}\|f_2\|^{q_2}_{L^{p_2}}dt.$$

From \eqref{cond-pq}, we know $$\frac{p_1}{q_1}\cdot\frac{q-q_1}{p-p_1}=\frac{p_2}{q_2}\cdot \frac{q-q_2}{p-p_2} < 0.$$
Choose
$
\sigma=\frac{p_1}{q_1}\cdot\frac{q-q_1}{p-p_1}=\frac{p_2}{q_2}\cdot \frac{q-q_2}{p-p_2},
$
and let  $k_1=\frac{q_1}{p_1}<0$ and $k_2=\frac{q_2}{p_2}<0.$
We have
\begin{equation*}
p_1+\frac{q-q_1}{\sigma k_1}=p_2+\frac{q-q_2}{\sigma k_2}=p.
\end{equation*}

%For $p_1< p_2,$
Let $\gamma=(\frac{t}{A})^\sigma$  with $A$ being a  constant to be  specified later.
From the reversed Minkowski inequality, we have
 %${\bf~(Note ~f_1=f~ if~|f|\le\gamma,~\Longrightarrow~f_1=f~if~A|f|^\frac1\sigma\le t,~ otherwise,f_1=0)~}$
\begin{eqnarray*}
V_1^\frac{1}{k_1} %&=& \big(\int_0^\infty t^{q-q_1-1}\|f_1\|^{q_1}_{L^{p_1}}dt\big)^\frac{1}{k_1}\\
&=& \big[\int_0^\infty\big(\int_{\mathbb{R}^n} t^{\frac{q-q_1-1}{k_1}} |f_1(x)|^{p_1}dx\big)^{k_1}dt\big]^\frac{1}{k_1}\\
&\ge&\int_{\mathbb{R}^n} \big(\int_0^{A|f(x)|^\frac1\sigma} t^{q-q_1-1} |f_1(x)|^{p_1k_1}dt\big)^\frac{1}{k_1}dx\\
%&=&\int_{\mathbb{R}^n} \big(\int_{A|f(x)|^\frac1\sigma}^\infty t^{q-q_1-1} |f(x)|^{p_1k_1}dt\big)^\frac{1}{k_1}dx\\ %\quad(q-q_1<0)\\
&=&\big(\frac{A^{q-q_1}}{q-q_1}\big)^\frac{1}{k_1}\int_{\mathbb{R}^n} |f(x)|^{p_1+\frac{q-q_1}{\sigma k_1}} dx.
\end{eqnarray*}
That is
\begin{equation}\label{m-8}
V_1\leq\frac{A^{q-q_1}}{q-q_1}\cdot \big(\int_{\mathbb{R}^n} |f(x)|^{p} dx\big)^{k_1}.
\end{equation}
For $V_2$, we have
\begin{eqnarray*}
V_2^\frac{1}{k_2}
%&=&\int_0^\infty t^{q-q_2-1}\|f_2\|^{q_2}_{L^{p_2}}dt\\
&=& \big[\int_0^\infty\big(\int_{\mathbb{R}^n}t^{\frac{q-q_2-1}{k_2}} |f_2(x)|^{p_2}dx\big)^{k_2}dt\big]^\frac{1}{k_2}\\
&\ge&\int_{\mathbb{R}^n}\big(\int_{A|f(x)|^\frac1\sigma}^\infty t^{q-q_2-1} |f_2(x)|^{p_2k_2}dt\big)^\frac{1}{k_2}dx\\
%&=&\int_{\mathbb{R}^n} \big(\int_0^{A|f(x)|^\frac1\sigma}  t^{q-q_2-1} |f(x)|^{p_2k_2}dt\big)^\frac{1}{k_2}dx\quad( \mbox{note} \quad q-q_2>0)\\
&=&\big(\frac{A^{q-q_2}}{q_2-q}\big)^\frac{1}{k_2}\int_{\mathbb{R}^n} |f(x)|^{p_2+\frac{q-q_2}{\sigma k_2}} d x.
\end{eqnarray*}
Thus,
\begin{equation}\label{m-9}
V_2\leq \frac{A^{q-q_2}}{q_2-q} \cdot \big(\int_{\mathbb{R}^n} |f(x)|^{p} dx\big)^{k_2}.
\end{equation}

Substituting \eqref{m-8} and \eqref{m-9} into \eqref{m-7}, we have
\begin{equation*}
\|Tf\|^q_{L^q}\leq C(V_1+V_2)
\leq C(A^{q-q_1} \|f\|^{pk_1}_{L^p}+ A^{q-q_2} \|f\|^{pk_2}_{L^p}).
\end{equation*}
Choosing suitable $A>0$ such that
\begin{equation*}
A^{q-q_1}\|f\|^{pk_1}_{L^p}=A^{q-q_2}\|f\|^{pk_2}_{L^p}. %\Rightarrow A=\|f\|^{\frac{p(k_2-k_1)}{q_2-q_1}}_{L^p},
\end{equation*}
We obtain
\begin{equation*}
\|Tf\|^q_{L^q}\leq C\|f\|^q_{L^p}.
\end{equation*}

\hfill$\Box$

\medskip

%-------------------------------------------------------------------------------------------------------------------------------------
\subsection{\textbf{The rough reversed HLS inequality}}\label{subsection 2.2}

%\section{\textbf{The extension of Hardy-Littlewood-Sobolev inequality} \label{Section 3}}
We are now ready to prove the reversed HLS  inequality with a rough constant (Theorem \ref{ext-HLS}).

%For $\alpha \in (0,\infty),$ define the classic singular integral operator on $\mathbb{R}^n$  by
%\begin{equation} \label{sig-ope-3}
%I_\alpha f(x)=\int_{\mathbb{R}^n}\frac{f(y)}{|x-y|^{n-\alpha}}dy, \quad ~~~~~~~~~~\forall x\in\mathbb{R}^n.
%\end{equation}
The reversed HLS  inequality with a rough constant can be derived from the following proposition.
\begin{prop}\label{HLS-equivalent} For any $1\le n<\alpha,$ $\frac n\alpha<p<1$ and $q$ given by
\begin{equation}\label{HlS-exp}
\frac1q=\frac1p-\frac\alpha n,
\end{equation} there exists a constant $C=C(n,\alpha,p)>0$, such that for all nonnegative $f\in L^p(\mathbb{R}^n)$,
%{\bf no need for nonnegativity. MJ 3-30-2012} Wrong!, 8-05-2012
\begin{equation}\label{HLS-ineq-2-1}
||I_\alpha f||_{L^q} \ge C||f||_{L^p}.
\end{equation}
%and the equality holds for certain extremal functions. Moreover, all extremal functions are radially symmetric (with respect to some points).
\end{prop}
% The equivalence  of  inequalities \eqref{HLS-ineq-1} and \eqref{HLS-ineq-2} is obvious. In fact,  i
 It follows from \eqref{HLS-ineq-2-1} and the reversed H\"older inequality that for nonnegative functions $f\in L^p$ and $g\in L^t$,
$$
<I_\alpha f, g> \ge C||f||_{L^p} \cdot ||g||_{L^t},$$
where $t=q'=\frac q{q-1}$ (thus $t\in(0,1));$ Which yields: for $\alpha>n$,
$$
|\int_{\mathbb{R}^n} \int_{\mathbb{R}^n} f(x)|x-y|^{\alpha-n} g(y) dx dy|\ge C||f||_{L^p}||g||_{L^t},
$$
where $t, \, p$ satisfy
$$
1-\frac 1t=\frac 1p-\frac \alpha n
\Leftrightarrow \frac 1 p+ \frac 1 t+\frac {n-\alpha} n=2.$$

\textbf{Proof of Proposition \ref{HLS-equivalent}.} For $1\le n<\alpha,$ $p\in(\frac n\alpha,1) $ and $q$ given by  (\ref{HlS-exp}), we first prove
\begin{equation}\label{HD-1-1}
\|I_\alpha f\|_{L^q_W}\ge C(n,\alpha,p)\|f\|_{L^p}
\end{equation}
for some constant $C(n,\alpha,p)$. That is,
we need  to show that there is a constant $C(n,\alpha, p)>0,$ such
that
\begin{equation}\label{HD-1}
meas\{x:|I_\alpha f(x)|<\tau\}\le\big(C(n,\alpha, p)\frac{\|f\|_{L^p}}\tau\big)^q, ~~~~~\forall f\in
L^p(\mathbb{R}^n) ~\mbox{and} ~ f(x)\ge 0, \ ~\forall~ \tau>0.
\end{equation}
Inequality (\ref{HD-1-1}) then implies \eqref{HLS-ineq-2-1} via the new Marcinkiewicz interpolation (Proposition \ref{Mar-1}).
% that
%\begin{equation*}\label{HD-1-2}
%\|I_\alpha f\|_{L^q}\ge C_1(n,\alpha,p)\|f\|_{L^p}.
%\end{equation*}
%or even slight stronger inequality
%$$%\begin{equation}\label{HD-1-3}
%\|I_\alpha f\|_{L^q( \mathbb{R}^n)}\ge C_2(n,\alpha,p)\|f\|_{L^{p,q}(
%\mathbb{R}^n)}
%$$%\end{equation}
% for $p\in (\frac {n}{\alpha},1)$ and $q$ being given by (\ref{HlS-exp}).

For any $\rho>0$, define
 \[I_{\alpha, \rho}^1 f(x)=\int_{|y-x|\le \rho}\frac{f(y)}{|x-y|^{n-\alpha}}dy,
  \]
  and
   \[I_{\alpha, \rho}^2 f(x)=\int_{|y-x|> \rho}\frac{f(y)}{|x-y|^{n-\alpha}}dy.
  \]
Note both $I_{\alpha, \rho}^1$ and $I_{\alpha, \rho}^2$ map nonnegative functions to nonnegative functions. Thus, for any $\tau>0$,
\begin{equation}\label{HD-2}
meas\{x: I_\alpha f(x)<2\tau\}\le meas\{x: I_{\alpha,\rho}^1 f (x)<\tau\}+meas\{x: I_{\alpha,\rho}^2f(x)<\tau\}.
\end{equation}
 We note that  it suffices to prove  inequality \eqref{HD-1} with $2\tau$ in place of $\tau$ in the left side of the
inequality, and we can further assume $\|f\|_{L^p}=1.$

Using the converse Young's inequality,  we have
\begin{eqnarray*}
\|I_{\alpha, \rho}^1 f \|_{L^{r_1}}&\ge&\big(\int_{\mathbb{R}^n}
\big(\frac{\chi_\rho(|x-y|)}{|x-y|^{n-\alpha}}\big)^{t_1}dy\big)^{\frac{1}{t_1}}\|f\|_{L^p}\\
&=:&D_1.
\end{eqnarray*}
where $\frac1p+\frac1{t_1}=1+\frac1{r_1}$ with $t_1\in (\frac{n}{n-\alpha}, 0) ,r_1<0,$ $\chi_\rho(x)=1$ for $|x|\le \rho$ and $\chi_\rho(x)=0$ for $|x|>\rho$, and
\begin{eqnarray*}
D_1&=&\big(\int_{B_\rho(x)}
 \frac{1}{|x-y|^{(n-\alpha){t_1}}} dy\big)^{\frac{1}{t_1}}=C_1(n,\alpha)\rho^\frac{ n-(n-\alpha){t_1} }{t_1}.
\end{eqnarray*}
It follows that
\begin{eqnarray}\label{HD-3}
meas\{x:|I_{\alpha,\rho}^1 f|<\tau\}\le\frac{\|I_{\alpha,\rho}^1 f \|^{r_1}_{L^{r_1}}}{\tau^{r_1}}\le \frac{C_2(n,\alpha)\rho^{\frac{r_1[n-(n-\alpha){t_1}]}{t_1}}}{\tau^{r_1}}.
\end{eqnarray}
On the other hand, by the converse Young's inequality, we have
\begin{eqnarray*}
\|I_{\alpha, \rho}^2 f\|_{L^{r_2}}
&\ge&\big(\int_{\mathbb{R}^n}\big(\frac{1-\chi_\rho(|x-y|)}{|x-y|^{n-\alpha}}\big)^{t_2}dy\big)^\frac1{t_2}\|f\|_{L^p}\nonumber\\
&=:&D_2,
\end{eqnarray*} where $\frac1p+\frac1{t_2}=1+\frac1{r_2}$ with $t_2<\frac{n}{n-\alpha},r_2<0.$  Easy to see: $r_2<\frac{np}{n-\alpha p}<r_1.$  Also,
\begin{eqnarray*}
D_2&=&\big(\int_{\mathbb{R}^n\backslash B_\rho(x)}
\big(\frac{1}{|x-y|^{n-\alpha}}\big)^{t_2}dy\big)^{\frac{1}{t_2}}=C_3(n,\alpha)\rho^{\frac{n-(n-\alpha)t_2}{t_2}}.
\end{eqnarray*} It follows that
\begin{eqnarray}\label{HD-3-1}
meas\{x:|I_{\alpha,\rho}^2 f|<\tau\}\le\frac{\|I_{\alpha,\rho}^2 f \|^{r_2}_{L^{r_2}}}{\tau^{r_2}}\le \frac{C_4(n,\alpha)\rho^{\frac{r_2[n-(n-\alpha){t_2}]}{t_2}}}{\tau^{r_2}}.
\end{eqnarray}
Bringing \eqref{HD-3} and \eqref{HD-3-1} into \eqref{HD-2}, we have
\begin{eqnarray*}
meas\{x:|I_\alpha f|<2\tau\}&\le& \frac{C_2(n,\alpha)\rho^{\frac{r_1[n-(n-\alpha){t_1}]}{t_1}}}{\tau^{r_1}}+\frac{C_4(n,\alpha)\rho^{\frac{r_2[n-(n-\alpha){t_2}]}{t_2}}}{\tau^{r_2}}.
\end{eqnarray*}

Now, choosing $\rho=\tau^\frac{p}{p\alpha-{n}}$, we have

\begin{eqnarray*}
\frac{p\,r_1}{p\alpha-{n}}\cdot[\frac n{t_1}-(n-\alpha)]-r_1
&=&\frac{p r_1}{p\alpha-n}\big[n\big( 1+\frac1{r_1}-\frac1p\big)-n+\alpha\big]-r_1\\
&=&\frac{p r_1}{p\alpha-n}\big[\frac{np-r_1n+pr_1\alpha}{r_1p}\big]-r_1\\
%&=&\frac{np-rn+pr\alpha}{p\alpha-n}-r\\
&=&-\frac{np}{n- p\alpha}\\
&=&-q.
\end{eqnarray*}
Similarly, $\frac{p\,r_2}{p\alpha-{n}}\cdot[\frac n{t_2}-(n-\alpha)]-r_2=-\frac{np}{n- p\alpha}=-q.$
%Let
%\[q= \frac{np}{n -p\alpha},
%\] that is,
%\[\frac 1q=\frac1p-\frac{\alpha}{n},
%\]
we then obtain (\ref{HD-1}). %Thus we complete the proof.
\hfill $\Box$

%checked 8-05-2012
%------------------------------------------------------------------
\section{\textbf{Existence and classifications of extremal functions for sharp inequalities}}\label{Section 4}We shall discuss the sharp form of the reversed HLS and prove Theorem \ref{ext-HLS-exis} in this section.

\subsection{Existence of extremal functions}
  In the case of $p=t$, we are able to show the existence of extremal functions.
 In this regard, it is relatively easier to state the sharp form of the inequality on the standard sphere $\mathbb{S}^n.$

Let $\mathcal{S}:\, x \in \mathbb{R}^n\to \xi\in \mathbb{S}^n\backslash(0,0,\cdots,-1)$ be the inverse of a stereographic projection, defined by
\begin{eqnarray*}
\xi^j:=\frac{2x^j}{1+|x|^2},\quad\text{for}~j=1,2,\cdots,n;\quad\xi^{n+1}:=\frac{1-|x|^2}{1+|x|^2}.
\end{eqnarray*}
Easy to check (or, see e.g.  \cite {Lieb1983, LL2001}): for  $x,y\in \mathbb{R}^n,\xi\in \mathbb{S}^n,$
\begin{eqnarray*}
|\mathcal{S}(x)-\mathcal{S}(y)|=\big[\frac{4|x-y|^2}{(1+|x|^2)(1+|y|^2)}\big]^\frac12, ~~~~~~~~~ \quad
d\xi=\big(\frac2{1+|x|^2}\big)^ndx.
\end{eqnarray*}
The area of the unit sphere in $\mathbb{R}^{n+1}$ is given by
\begin{eqnarray*}
|\mathbb{S}^n|:=\int_{\mathbb{S}^n}d\xi= {2\pi^{\frac{n+1}2}}{\big[\Gamma\big(\frac{n+1}2\big)\big]^{-1}}=2^{n}\pi^{\frac{n}2}\frac{\Gamma(n/2)}{\Gamma(n)}.
\end{eqnarray*}

For $1\le n<\alpha$, let $p=\frac{2n}{n+\alpha},q=\frac{2n}{n-\alpha}$ throughout whole subsection 3.1. For any $F(\xi)\in L^p(\mathbb{S}^n),$ let $x=\mathcal{S}^{-1}(\xi) \in \mathbb{R}^n,$ and  define
$f(x):=\big(\frac2{1+|x|^2}\big)^{\frac {n+\alpha}2}F(\xi).$  Also recall
\[\tilde{I}_\alpha F(\xi)=\int_{\mathbb{S}^n}\frac{F(\eta)}{|\xi-\eta|^{n-\alpha}}d\eta.\]
Direct computation yields
\begin{eqnarray*}
\int_{\mathbb{S}^n} |F(\xi)|^p d\xi=\int_{\mathbb{R}^n}\big(\frac2{1+|x|^2}\big)^{-n}|f(\mathcal{S}^{-1}(\xi))|^p\big(\frac2{1+|x|^2}\big)^{n}dx
=\int_{\mathbb{R}^n}|f(x)|^pdx;
\end{eqnarray*}
And
\begin{eqnarray*}
\|\tilde{I}_\alpha F\|^q_{L^q(\mathbb{S}^n)}&=&
\int_{\mathbb{S}^n}\big(\int_{\mathbb{S}^n}\frac{F(\eta)}{|\xi-\eta|^{n-\alpha}}d\eta\big)^qd\xi\\
&=&\int_{\mathbb{R}^n}\big(\int_{\mathbb{R}^n}\big(\frac2{1+|y|^2}\big)^{-\frac{n+\alpha}2}f(\mathcal{S}^{-1}(\eta))\big[\frac{4|x-y|^2}{(1+|x|^2)(1+|y|^2)}\big]^{-\frac{n-\alpha}2} \\
& &\times\big(\frac2{1+|y|^2}\big)^{n}dy\big)^q \big(\frac2{1+|x|^2}\big)^{n}dx\\
%&=& \|{I_\alpha} F\|^q_{L^q(\mathbb{R}^n)}
%&=&\int_{\mathbb{R}^n}\big(\int_{\mathbb{R}^n}\big(\frac2{1+|y|^2}\big)^{-\frac np+n-\frac{n-\alpha}2} \frac{f(y)}{|x-y|^{n-\alpha}}dy\big)^q \big(\frac2{1+|x|^2}\big)^{n-\frac{q(n-\alpha)}2}dx\\
&=&\int_{\mathbb{R}^n}\big(\int_{\mathbb{R}^n}\frac{f(y)}{|x-y|^{n-\alpha}}dy\big)^q dx.
\end{eqnarray*}

Thus
\begin{eqnarray}\label{mj7-22-1}
\|F\|_{L^p(\mathbb{S}^n)}=\|f\|_{L^p(\mathbb{R}^n)},\quad \text{and} \quad \, ~\|\tilde{I}_\alpha F\|_{L^q(\mathbb{S}^n)}=\|I_\alpha f\|_{L^q(\mathbb{R}^n)}.
\end{eqnarray}

We hereby have an equivalent sharp reversed HLS
inequality \eqref{HLS-ineq-1-sph} for all nonnegative $F(\xi), \ G(\xi) \in L^p(\mathbb{S}^n)$.
% or equivalently,
%\begin{equation}\label{HLS-sph-2}
%||\tilde I_\alpha F||_{L^q} \ge N^*(n, \alpha)||F||_{L^p}, \ \ \, \, \ \forall F(\xi)  \in L^p(\mathbb{S}^n) \ \ \mbox{and} \ \ F\ge 0.
%\end{equation}
The sharp constant to inequality (\ref{HLS-ineq-2}) is classified by
\begin{eqnarray}\label{mj-7-20-1}
 N^{*} ({n,\alpha})&=&\inf\{\|\tilde I_\alpha F\|_{L^q(\mathbb{S}^n)}\,:\, F\ge 0, \, \, \, \,\|F\|_{L^p( \mathbb{S}^n)}=1\}\nonumber\\
 &=&\inf\{\|I_\alpha f\|_{L^q(\mathbb{R}^n)}\,:\, f\ge 0, \, \, \, \,\|f\|_{L^p( \mathbb{R}^n)}=1\}.
\end{eqnarray}

We remark that we only need to show that  sharp inequality (\ref{HLS-ineq-2}) holds for all nonnegative $F \in L^1( \mathbb{S}^n)$. In fact, if for all nonnegative $F, G \in L^1( \mathbb{S}^n)$,
$$
\int_{\mathbb{S}^n} \int_{\mathbb{S}^n} F(\xi)|\xi-\eta|^{\alpha-n} G(\eta) dS_\xi dS_\eta\ge N^*(n,\alpha)||F||_{L^p(\mathbb{S}^n)}||G||_{L^p(\mathbb{S}^n)},
$$
then for any nonnegative $u, v \in L^p( \mathbb{S}^n)$, we consider $u_A=\min(u, A)\in L^1( \mathbb{S}^n)$ and  $v_A=\min(v, A) \in L^1( \mathbb{S}^n)$, thus
$$
\int_{\mathbb{S}^n} \int_{\mathbb{S}^n} u_A(\xi)|\xi-\eta|^{\alpha-n} v_A(\eta) dS_\xi dS_\eta\ge N^*(n,\alpha)||u_A||_{L^p(\mathbb{S}^n)}||v_A||_{L^p(\mathbb{S}^n)}.
$$
Sending $A \to \infty$, we obtain via the monotone convergence theorem the desired sharp inequality
  for $u, v \in L^p( \mathbb{S}^n)$.

  Since we are dealing with a reserved inequality, the usual density argument does not work here. More specifically,  even one can prove inequality \eqref{HLS-ineq-1} for all $f, \ g \in C_0^\infty(\mathbb{R}^n)$, it is not obvious that the inequality also holds for general function $f \in L^p(\mathbb{R}^n)$ and $g \in L^t(\mathbb{R}^n)$. So we need to establish the following density lemma on $\mathbb{S}^n.$

  %{\bf (This is one of the main reasons that we can not prove the sharp reversed inequality for $p\ne 2n/(n+\alpha)$.)}

\begin{lm}\label{lem3.1}{\bf (Density Lemma)} Let $F(\xi) \in  L^1( \mathbb{S}^n)$ be a nonnegative function with $\|F\|_{L^p( \mathbb{S}^n)}=1$. For any $\epsilon>0$, there is a nonnegative $G(\xi)\in C^0( \mathbb{S}^n)$, such that
$$
||F-G||_{L^p( \mathbb{S}^n)}+\big{|}||\tilde I_\alpha F||_{L^q( \mathbb{S}^n)}- ||\tilde I_\alpha G||_{L^q( \mathbb{S}^n)}\big {|}<\epsilon.
$$
\end{lm}
\begin{proof} Let $\{G_i\}_{i=1}^\infty$ be a sequence of nonnegative, continuous functions such that $||G_i-F||_{L^1( \mathbb{S}^n)} \to 0$ as $i\to \infty$. Then, for any $\xi\in \mathbb{S}^n,$ as $i\to \infty$,
\begin{eqnarray*}
|\tilde I_\alpha G_i(\xi)-\tilde I_\alpha F(\xi)|   & \le& \int_{\mathbb{S}^n} |G_i(\eta)-F(\eta)|\cdot |\xi-\eta|^{\alpha-n} d\eta \\
&\le& C \int_{\mathbb{S}^n} |G_i(\eta)-F(\eta)| d\eta \to 0.
\end{eqnarray*}
Since $||F||_{L^p( \mathbb{S}^n)}=1$, we know that $||F||_{L^1( \mathbb{S}^n)}\ge C>0$. This implies that $\tilde I_\alpha F(\xi)$ is a continuous and positive function. Thus $\tilde I_\alpha G_i(\xi)\ge C>0$ for large $i$. From the dominant convergent theorem, we have
$$
\liminf_{i\to\infty}\int_{\mathbb{S}^n}|\tilde I_\alpha G_i|^q =\int_{\mathbb{S}^n}|\tilde I_\alpha F|^q.
$$
Lemma \ref{lem3.1} then follows from the above.
\end{proof}

%\begin{rem}\label{rem_lemma3.1}
%We tend to believe that condition $F\in L^1( \mathbb{S}^n)$ can not be dropped in Lemma \ref{lem3.1}. If it can be dropped, then inequality \eqref{HLS-sph-2} holds for  $N^{*} ({n,\alpha})= N_1^{*} ({n,\alpha}).$
%\end{rem} Remark removed on 4-18-2013

\smallskip

Next, we prove that the infimum in \eqref{mj-7-20-1} is attained. Due to Lemma \ref{lem3.1} and the remark before it, we can choose $\{F_j\}_{j=1}^\infty \in C^0 (\mathbb{S}^n)$ to be a  nonnegative minimizing sequence with $\|F_j\|_{L^p}=1$. Let $f_j(x):=\big(\frac2{1+|x|^2}\big)^{\frac {n+\alpha}2}F_j(\mathcal{S}(x)),$ then $\{f_j\}_{j=1}^\infty \in C^0 (\mathbb{R}^n)$ is the  nonnegative corresponding minimizing sequence with $\|f_j\|_{L^p}=1$ on $\mathbb{R}^n.$

For a given nonnegative measurable function $u(x)$ on $\mathbb{R}^n$ decaying at infinity, we can define its radially symmetric, non-increasing rearrangement
function $u^*$.
%{\bf precise condition on $f(x)$ at infinity is needed! Note the density argument does not work since the inequality is reversed.}
 $u^*(x)$ is a nonnegative lower-semicontinuous
 function and has the same distribution as $u$.
%It satisfies the following important Riesz rearrangement
%inequality \cite{Lieb2001} (p. 87): for any nonnegative measurable functions $u, v,$ and $w$ on
%$\mathbb{R}^n$, we have
%\begin{eqnarray*}\int_{\mathbb{R}^n}\int_{\mathbb{R}^n}u(x)v(x-y)w(y)dydx\le\int_{\mathbb{R}^n}\int_{\mathbb{R}^n}u^*(x)v^*(x-y)w^*(y)dydx.
%\end{eqnarray*}
Define $v_*=((v^{-1})^*)^{-1}$, then $v_*$  is radially symmetric, increasing rearrangement
function. It is known (see, e.g. the proof of Proposition $9$ in  Brascamp and Lieb \cite{BL1976}) that
\begin{eqnarray*}
\int_{\mathbb{R}^n}\int_{\mathbb{R}^n}u(x)w(x-y)v(y)dydx &\ge\int_{\mathbb{R}^n}\int_{\mathbb{R}^n}u^*(x)w^*(x-y)v_*(y)dydx\\
&=\int_{\mathbb{R}^n}\int_{\mathbb{R}^n}u^*(x)v_*(x-y)w^*(y)dydx.
\end{eqnarray*}

Suppose that $\|w\|_{L^{q'}(\mathbb{R}^n)}=\|w^*\|_{L^{q'}(\mathbb{R}^n)}=1$ for $0<q'<1$. Then
for $q<0$ and $q'=q/(q-1)$, we have
\begin{eqnarray}\label{cov_prop1}
\nonumber \|u*v\|_{L^q}&=&\inf_{\|w\|_{L^{q'}}=1}\int_{\mathbb{R}^n}
\int_{\mathbb{R}^n}u(x)v(x-y)w(y)dydx\nonumber \\
&\ge&\inf_{\|w*\|_{L^{q'}}=1}\int_{\mathbb{R}^n}\int_{\mathbb{R}^n}u^*(x)v_*(x-y)w^*(y)dydx\nonumber \\
&\ge&\inf_{\|w*\|_{L^{q'}}=1}\int_{\mathbb{R}^n}\big(\int_{\mathbb{R}^n}(u^*(x)v_*(x-y))^qdy\big)^\frac1q\big(\int_{\mathbb{R}^n}(w^*(y))^{q'}dy\big)^\frac1{q'}dx\nonumber \\
&=&\|u^**v_*\|_{L^q}.
\end{eqnarray}
%Moreover, if $v$ is nonnegative, radially symmetric, and strictly increasing in the
%radial direction, $u$ is nonnegative, $q<0$ and
%\begin{equation}\label{cov_prop2}
%\|u*v\|_{L^q}=\|u^**v \|_{L^q}<\infty,
%\end{equation}then $u(x)=u^*(x-x_0)$ for some $x_0\in \mathbb{R}^n$. See, for example,  Section $3.9$, $P_{93}$ in \cite{LL2001}. This indicates that if $u(x)$ is a minimizer to $\inf\{\|I_\alpha f\|_{L^q(\mathbb{R}^n)}\,:\, f\ge 0, \, \, \, \,\|f\|_{L^p( \mathbb{R}^n)}=1\},$  $u(x)$ is radially symmetric about some point.

Let $f_j^*$ be the non-increasing radial symmetric rearrangement of $f_j.$
Since
\[\|f_j^*\|_{L^p}=\|f_j\|_{L^p}=1,
\]
and
\begin{eqnarray*}
\|I_\alpha
(f_j)\|^q_{L^q}&=&\int_{\mathbb{R}^{n}}\int_{\mathbb{R}^{n}}
\big(\int_{\mathbb{R}^{n}}\frac{f_j(y)}{|x-y|^{n-\alpha}}dy\big)^qdx\\
&\le&\int_{\mathbb{R}^{n}}
\big(\int_{\mathbb{R}^{n}}\frac{f_j^*(y)}{|x-y|^{n-\alpha}}dy\big)^qdx ~ \ \ \ \ ~ ~ ~ ~ (\mbox{by} \ \ (\ref{cov_prop1}))\\
&=&\|I_\alpha (f_j^*)\|^q_{L^q},
\end{eqnarray*}
we know that $\{f_j^*\}_{j=1}^\infty$ is also a minimizing sequence.  Without loss of generality, we can  assume that $\{f_j\}_{j=1}^\infty$ is a
nonnegative radially symmetric and  non-increasing minimizing sequence.

For $\alpha \in (0, n)$,
to avoid that  $f_j$ converges to a trivial function,  Lieb modified his maximizing sequence via a technical lemma (Lemma 2.4 in Lieb \cite{Lieb1983}). In our case, we need to  modify the minimizing sequence in a similar way so that both $I_\alpha f_j$ and $f_j$ will stay away from the trivial function via the following lemma.

\begin{lm}\label{Lieb-lm-1}
Let $p_1\in(0, 2n/(n+\alpha))$, and $s \in (\frac{n}{n-\alpha}, 0)$ be two parameters satisfying $\frac1{p_1}+\frac1s-1=\frac{n-\alpha}{2n}.$
Suppose that $f\in L^p(\mathbb{R}^n)$  is a  nonnegative, radially symmetric function satisfying $f(|x|)\le\varepsilon|x|^{-\frac np}$ for all $|x|>0$. Then, there exists a constant $C_n$ independent of $f$ and $\varepsilon$ such that
\begin{equation}\label{L-1}
\|I_\alpha f\|_{L^q(\mathbb{R}^n)}\ge C_n\varepsilon^{1-\frac p{p_1}}\|f\|_{L^p(\mathbb{R}^n)}^{\frac p{p_1}}.
\end{equation}
\end{lm}
\begin{proof} Our proof is similar to that of Lemma 2.4 in Lieb \cite{Lieb1983}.

Define $F\,:\, \mathbb{R}\to\mathbb{R}$ by
\[ F(u)=e^{\frac{un}p}f(e^u).
\]We can easily see that
\[(n\omega_n)^\frac1p\|F\|_{L^p(\mathbb{R})}=\|f\|_{L^p(\mathbb{R}^n)}, \quad\text{and}\quad\|F\|_{L^\infty(\mathbb{R})}\le \varepsilon,
\] where $\omega_n=\frac{2\pi^\frac n2}{n}\Gamma(\frac n2)$ denotes the volume of the $n$-dimensional unit ball.
Define $h=I_\alpha f$. Easy to see that $h$ is radially symmetric. Define $H\,:\, \mathbb{R}\to\mathbb{R}$  by
\[H(u)=e^{\frac{un}q}h(e^u).
\]Then
\[(n\omega_n)^\frac1q\|H\|_{L^q(\mathbb{R})}=\|h\|_{L^q(\mathbb{R}^n)}.
\]

By integrating $dx$ over angles in $\mathbb{R}^n$, an explicit form for $H$ can be obtained as follows.
\begin{eqnarray*}
H(u)=\int_{-\infty}^{+\infty}L_n(u-v)F(v)dv,
\end{eqnarray*}where
\begin{eqnarray*}
L_n(u)&=&\big(\frac12\big)^\frac{n-\alpha}2e^{u(\frac nq-\frac{n-\alpha}2)}Z_n(u),\\
Z_n(u)&=&
\begin{cases}
(n-1)\omega_{n-1}\int_{0}^\pi(\cosh u-\cos\theta)^\frac{\alpha-n}2(\sin\theta)^{n-2}d\theta,\quad&n\ge2,\\
(\cosh u+1)^\frac{\alpha-n}2+(\cosh u-1)^\frac{\alpha-n}2,\quad&n=1.
\end{cases}
\end{eqnarray*}
We have $L_n\in L^s(\mathbb{R})$ for any given $s<0$.

Now, by the converse Young inequality (Lemma 2.2), for given $p_1\in(0, 2n/(n+\alpha))$, and $s \in (\frac{n}{n-\alpha}, 0)$  satisfying $\frac1{p_1}+\frac1s-1=\frac{n-\alpha}{2n}=\frac 1 q,$ we have
\begin{eqnarray}\label{L-2}
\|H\|_{L^q(\mathbb{R})}\ge\|L_n\|_{L^s(\mathbb{R})}\|F\|_{L^{p_1}(\mathbb{R})}.
\end{eqnarray}
On the other hand, since $p_1<p<1$,  we have
\begin{eqnarray*}
\|F\|_{L^{p_1}(\mathbb{R})}&=&\big(\int_{-\infty}^{+\infty}|F(v)|^p|F(v)|^{p_1-p}dv\big)^\frac1{p_1}\\
&\ge&\|F\|^{1-\frac p{p_1}}_{L^{\infty}(\mathbb{R})}\|F\|^\frac p{p_1}_{L^{p}(\mathbb{R})}\\
&\ge&\varepsilon^{1-\frac p{p_1}} \|F\|^\frac p{p_1}_{L^{p}(\mathbb{R})}.
\end{eqnarray*}
%{\bf ($\|F\|^\frac p{p_1}_{L^{p_1}}$ changed to $\|F\|^\frac p{p_1}_{L^{p}}$. 5-17-2012)}
Combining the above with  \eqref{L-2}, we obtain \eqref{L-1}.
\end{proof}

For convenience, denote $e_1=(1,0,\cdots,0,0)\in\mathbb{R}^n$, and  define
\begin{equation*}
a_j:=\sup_{\lambda>0} \lambda^{-\frac{n}p}f_j(\frac {e_1}\lambda).
\end{equation*}
Note that  for $y\in \mathbb{R}^{n}$,
\[f_j(y)=f_j(|y|e_1)=|y|^{-\frac{n}p}|y|^{\frac{n}p}f_j(|y|e_1)\le a_j|y|^{-\frac{n}p},
\]
and $\|I_\alpha f_j\|_{L^q} \to N^{*} ({n,\alpha}) <\infty.$ We know from Lemma \ref{Lieb-lm-1} that $a_j\ge 2 c_0>0.$

For any given nonnegative function $g(x)$ and $\lambda>0$, define $ g^{\lambda}(x) = {\lambda}^{-\frac{n}p}g(\frac {x}{\lambda})$. Easy to check that
$$
%\begin{equation}\label{mj-5-15-1}
I_\alpha g^\lambda (x)= \lambda^{\alpha-\frac n p}(I_\alpha g)(\frac x \lambda);
$$
%\end{equation}
and
\begin{equation}\label{mj-5-17-1}
\|g^{\lambda}\|_{L^p}=\|g\|_{L^p},
\quad \quad \|I_\alpha (g^{\lambda})\|_{L^q}=\|I_\alpha g\|_{L^q}.
\end{equation}

For each $j$, choose $\lambda_j$ so that $f_j^{\lambda_j}(e_1)\ge c_0.$ Due to \eqref{mj-5-17-1},
we know that $\{f_j^{\lambda_j}\}_{j=1}^\infty$ is also a minimizing sequence.
Therefore, we can further assume that there is a nonnegative,
radially symmetric and non-increasing minimizing sequence $\{f_j\}_{j=1}^\infty$ with $\|f_j\|_{L^p}=1$ and $f_j (e_1)\ge c_0$. Similar to Lieb's argument, we know, up to a subsequence, that $f_j \to f_\circ$ a.e. in $\mathbb{R}^{n}$.

Consider the corresponding minimizing sequence $F_j(\xi)= \big(\frac{1+|\mathcal{S}^{-1}(\xi)|^2}2\big)^{\frac {n+\alpha}2}f_j(\mathcal{S}^{-1}(\xi)),$ and $F_\circ(\xi)= \big(\frac{1+|\mathcal{S}^{-1}(\xi)|^2}2\big)^{\frac {n+\alpha}2}f_\circ (\mathcal{S}^{-1}(\xi)).$ We know $F_j(\xi)=F_j(\xi^{n+1}).$ Denote $\mathfrak{N}=(0, ..., 0, 1)$ as the north pole of the sphere, and $\xi_1=\mathcal{S}(e_1).$ So $F_j(\xi_1)\ge c_0$, and $F_j(\xi)\ge 2^{-(n+\alpha)/2} c_0$ for all $\xi$ in the geodesic ball $B_{r_0}(\mathfrak{N})$ where $r_0=dis(\xi_1, \mathfrak{N})$ on $\mathbb{S}^n$. Thus, there is a positive universal constant $C>0$, such that
\begin{equation}\label{lower-bound}
\tilde I_\alpha F_j(\xi)\ge C, \ \ \ \forall \ \xi \in \mathbb{S}^n.
\end{equation}
If $\tilde I_\alpha F_j(\xi) \to +\infty$ almost everywhere, then the dominant convergent theorem (using \eqref{lower-bound}) yields that $\lim_{j\to \infty}\int_{\mathbb{S}^n}|\tilde I_\alpha F_j(\xi)|^q=0.$  But $\lim_{j\to \infty}\int_{\mathbb{S}^n}|\tilde I_\alpha F_j(\xi)|^q=(N^*(n, \alpha))^{q}>0.$ Contradiction.
Thus for $\eta=(0,..., 0, \eta^{n+1})$ and $\eta^{n+1}\in (a, b)\subset (-1, 1)$, $\tilde I_\alpha F_j(\eta)<C(a, b)$ for certain constant $-1\le a<b\le 1$ and a constant $C(a, b)$ depending only on $a,\ b$. This yields
\begin{equation}\label{upper-bound}
\int_{\mathbb{S}^n}F_j \le C_{a, b}
\end{equation}
for some constant $C_{a, b}$ only depending on $a, \ b$. From \eqref{upper-bound} we know that sequence $\{\tilde I_\alpha F_j\}$ is uniformly bounded and equicontinuous on $\mathbb{S}^n$. Up to a subsequence, $\tilde I_\alpha F_j(\xi) \to L(\xi)\in C^0(\mathbb{S}^n).$ Using Fatou Lemma and the reversed HLS, we have, up to a further subsequence, for $m\in \mathbb{N}$, that
\begin{eqnarray*}
0&\ge&\big(\lim_{j\to \infty}\int_{\mathbb{S}^n}|\tilde I_\alpha F_j(\xi)-\tilde I_\alpha F_{j+m}(\xi)|^q  \big)^{1/q}\\
&\ge& C (\lim_{j\to \infty}||F_j-F_{j+m}||^q_{L^p})^{1/q}.
\end{eqnarray*}
Thus $||F_j-F_{j+m}||_{L^p}\to 0$. Since $F_j \to F_\circ$ pointwise, we know that $||F_j-F_\circ||_{L^p}\to 0$, thus $||F_\circ||_{L^p}=1.$

On the other hand, the dominant convergent theorem (using \eqref{lower-bound}) yields $\lim_{j\to \infty}\int_{\mathbb{S}^n}|\tilde I_\alpha F_j(\xi)|^q=\int_{\mathbb{S}^n}|L(\xi)|^q,$ and we know from Fatou Lemma that
$$
L(\xi)=\lim_{j\to \infty} \tilde I_\alpha F_j(\xi)\ge \tilde I_\alpha F_\circ(\xi).$$
It follows that
$$\big(\int_{\mathbb{S}^n}|L(\xi)|^q \big)^{1/q} \ge \big( \int_{\mathbb{S}^n}|\tilde I_\alpha F_\circ(\xi)|^q \big)^{1/q}.$$
Thus the $\inf\{\|\tilde I_\alpha F\|_{L^q(\mathbb{S}^n)}\,:\, F\ge 0, \, \, \, \,\|F\|_{L^p( \mathbb{S}^n)}=1\}$ is achieved by $F_\circ(\xi)$.

\begin{rem}\label{pointwise} It is not clear whether the pointwise convergence
\begin{equation}\label{t-power}
\tilde I_\alpha F_j(\xi) \to \tilde I_\alpha F_\circ(\xi) \ \ \ \ a.e. \mbox{on} \ \ \mathbb{S}^n
  \end{equation}
  is true or not. Even though we tend to believe this is the case, we do not know how to prove it. On the other hand,
  we point out here that a new phenomenon does arise while dealing with a concentrating minimizing sequence for $q<0$. We will show that without assuming that $f_j(e_1)\ge c_0$ for the corresponding minimizing sequence $\{f_j(x)\}$ defined on $\mathbb{R}^n$, the pointwise convergence
\eqref{t-power} may not be true. This is opposite  to the case for $q>0$, where the pointwise convergence \eqref{t-power} usually holds for extremal sequences.

In fact, for
$$f_j(x)=(\frac {\epsilon_j}{\epsilon_j^2+|x|^2})^{\frac {n+\alpha}2}$$
where $\epsilon_j \to 0$ as $ j \to \infty$, we know $f_j \to f_\circ=0$ a.e. in $\mathbb{R}^{n}.$ One may check directly that $I_\alpha f_j$ does not converge to $0$ a.e. in $\mathbb{R}^{n}$ (in fact, $I_\alpha f_j \to \infty$ a.e. in $\mathbb{R}^{n}$). This can also be observed from the reversed HLS inequality:
\begin{eqnarray}\label{mj-7-10}
\|I_\alpha|f_j-f_\circ|\|_{L^q}\ge C(n,\alpha,p)\|f_j-f_\circ\|_{L^{p}}.
\end{eqnarray}
if $I_\alpha f_j \to 0$ pointwise, from Fatou Lemma, we know that the left side in \eqref{mj-7-10}  will go to 0, but the right side $||f_j||_{L^p}=constant>0$.  Impossible.
%This also indicates that Lemma \ref{lem3.1} is essential in our proof for the existence of extremal functions.
\end{rem}

Let $F_\circ\in L^p(\mathbb{S}^{n})$
be a nonnegative minimizer. After normalization, we can assume  $||F_\circ||_{L^p}=1$. Easy to see  $\tilde I_\alpha F_\circ(\xi) \ge C>0$. Thus, for any positive smooth test function $\phi\in C^\infty(\mathbb{S}^{n})$, we have
\begin{equation}\label{mj7-11-2}
\int_{\mathbb{S}^n}F_\circ^{p-1}(\xi) \phi(\xi) d\xi \le C\int_{\mathbb{S}^n}\int_{\mathbb{S}^n} \frac{(\tilde I_\alpha F_\circ(\eta))^{q-1}}{|\eta-\xi|^{n-\alpha}}\phi(\xi)d\eta d\xi\le C_1<\infty.
\end{equation}
 Since $p<1$, we conclude that there is a positive constant $c_0>0$ such that $F_\circ(\xi)>c_0$ everywhere on $\mathbb{S}^n$. Thus
$F_\circ(\xi)$ is a weak positive solution to
\begin{equation}\label{Euler-equ-1}
F_\circ^{p-1}(\xi)=\int_{\mathbb{S}^n} \frac{(\tilde I_\alpha F_\circ(\eta))^{q-1}}{|\xi-\eta|^{n-\alpha}}d\eta,
\quad\quad \forall\, \xi\in \mathbb{S}^n.
\end{equation}

 To complete the proof of Theorem \ref{ext-HLS-exis}, we need to classify all positive solutions to \eqref{Euler-equ-1},  and to computer the best constant next.

 Let $f(x):=\big(\frac2{1+|x|^2}\big)^{\frac {n+\alpha}2}F_\circ(\mathcal{S}(x)),$ then $f(x)$ is a measurable positive function, satisfying:
\begin{equation}\label{Euler-equ-2}
f^{p-1}(x)=\int_{\mathbb{R}^n} \frac{(I_\alpha f(y))^{q-1}}{|x-y|^{n-\alpha}}dy,
\quad\quad \forall\, x \in \mathbb{R}^n.
\end{equation}

%~~~~~~~~~~~~~~~~~~~~~~~~~~~~~~~~~~~~~~~~~~~~~~~~~~~~~~~~~~~~~~~~~~~~~~~~~~~~

\subsection{Extremal functions and best constant}
 We will classify all positive, measurable solutions to equation \eqref{Euler-equ-2}  via the method of moving sphere for $ p=2n/(n+\alpha)$ and $q=2n/(n-\alpha)$, and compute the best constant $N^*(n, \alpha).$

For $R>0, x\in\mathbb{R}^n,$ denote
$$B_R(x)=\{y\in\mathbb{R}^n\,: \,|y-x|<R\},\quad\text{and}~ \Sigma_{x,R}=\mathbb{R}^n\backslash\overline{B_R(x)}.$$
For $x=0$, we write
$B_R=B_R(0),\Sigma_{R}=\Sigma_{0,R}$.

\subsubsection{Regularity\label{subsection 3.1}}
First, we  show that positive solutions to \eqref{Euler-equ-2}  are smooth except the case that the function $f(x)$ (thus $I_\alpha f(x)$) is infinity everywhere.
%In this part, we establish the regularity for all  positive solutions  to \eqref{Euler-equ-2}.
Throughout this subsection, we always assume that  $f$ is a positive measurable function satisfying \eqref{Euler-equ-2} such that  both $f$ and $ I_\alpha f \not\equiv \infty$.

Define $u(y)=f^{p-1}(y),$ $v(x)=I_\alpha f(x)$, $\theta=\frac1{p-1}<0 $ and
$\kappa=q-1<0.$ Then $u,v $ are also positive measurable functions and the single equation \eqref{Euler-equ-2} can be rewritten as an integral system
\begin{equation}\label{sys-1}
\begin{cases}
u(y)=\int_{\mathbb{R}^{n}}{|x-y|^{\alpha- n}} {v^{\kappa}(x)}dx,&\quad y\in\mathbb{R}^{n},\\
v(x)=\int_{\mathbb{R}^{n}}{|x-y|^{\alpha- n}} {u^{\theta}(y)}dy,&\quad x\in\mathbb{R}^{n}.
\end{cases}
\end{equation}
%with
%\begin{equation}
%\label{expon-1}
%\frac1{\kappa+1}+\frac1{\theta+1}=\frac{n-\alpha}{n}.
%\end{equation}

%{\bf Jingbo, if you copy other's proof, then you have to write down where, which paper you copied. For an important paper, and other will read in details, you have to be honest. MJ, 4-21-2012}

\begin{lm}\label{Reg-lm-1}
For  $1\le n<\alpha$ and $\theta,\kappa<0$, if $(u,v)$ is a pair of positive Lebesgue measurable solutions to \eqref{sys-1}, then
\begin{eqnarray*}
&(i)&~~\int_{\mathbb{R}^n}(1+|y|^{\alpha-n})u^{\theta}(y)dy<\infty, \quad\text{and}\quad \int_{\mathbb{R}^n}(1+|x|^{\alpha-n})v^{\kappa}(x)dx<\infty;\\
&(ii)& ~~a:=\lim_{|y|\to\infty}|y|^{n-\alpha}u(y)=\int_{\mathbb{R}^n}v^{\kappa}(x)dx<\infty, \\
\quad\quad& & b:=\lim_{|x|\to\infty}|x|^{n-\alpha}v(x)=\int_{\mathbb{R}^n}u^{\theta}(y)dy<\infty;\\
%even this $p$ copied from YY Li! 4-21-2012
&(iii) &\text{for~some~constants}~C_1,C_2>0,\\
\quad  & &\frac{1+|y|^{\alpha-n}}{C_1}\le u(y)\le C_1(1+|y|^{\alpha-n}),~~~~~~~~~~~~~~\forall\,y\in\mathbb{R}^n,\\
\quad  & &\frac{1+|x|^{\alpha-n}}{C_2}\le v(x)\le C_2(1+|x|^{\alpha-n}),~~~~~~~~~~~~~~\forall\,x\in\mathbb{R}^n.
\end{eqnarray*}
\end{lm}
\begin{proof} The proof is the same as that of Lemma 5.1 in Li \cite{Li2004}. We include details for the completion of the paper.

Since $u, \ v \not\equiv \infty$, we know that
\[meas\{y\in\mathbb{R}^n\,:\,u(y)<\infty \}>0, \quad\text{and}\quad meas\{x\in\mathbb{R}^n\,:\,v(x)<\infty \}>0.
\]
Thus, there exist $R>1$ and some measurable set $E$ such that
\[E\subset \{x\in \mathbb{R}^n\,:\,v(x)<R\}\cap B_R
\] with $|E|>\frac1R.$
It follows, for any $y\in\mathbb{R}^n$, that
\begin{eqnarray*}
u(y)&=&\int_{\mathbb{R}^n}|x-y|^{\alpha-n} v^{\kappa}(x)dx\\
&\ge&\int_{E}|x-y|^{\alpha-n} v^{\kappa}(x)dx\\
&\ge&R^{\kappa}\int_{E} |x-y|^{\alpha-n}dy.
\end{eqnarray*}  And then,
\begin{eqnarray*}
\lim_{|y|\to\infty}\frac{u(y)}{(1+|y|^{\alpha-n})}
&\ge&\lim_{|y|\to\infty}\frac{R^{\kappa}}{(1+|y|^{\alpha-n})}\int_{E} |x-y|^{\alpha-n}dx=CR^{\kappa-1}.
\end{eqnarray*}
This shows
\[{u(y)}\ge\frac{(1+|y|^{\alpha-n})}{C_1}.
\]
Similarly, for any $x\in\mathbb{R}^n$, we have
\[{v(x)}\ge\frac{(1+|x|^{\alpha-n})}{C_2}.
\]
This implies that the left hand side inequalities in $(iii)$ hold. %integrability for $v^\kappa$
% in $B_1(0). MJ 4-22-2012

On the other hand, for some $y_0\in\mathbb{R}^n $ with $1\le|y_0|\le2$,
\begin{eqnarray*}
\int_{\mathbb{R}^n}|x-y_0|^{\alpha-n}v^{\kappa}(x)dx
&=&u(y_0)<\infty;
\end{eqnarray*}
And, for some $x_0\in\mathbb{R}^n  $  with  $1\le|x_0|\le2$,
\begin{eqnarray*}
\int_{\mathbb{R}^n}|x_0-y|^{\alpha-n} u^{\theta}(y)dy
&=&v(x_0)<\infty.
\end{eqnarray*}
From the left hand side inequalities in $(iii)$ and the above, we obtain $(i)$.

For $|x|\ge1$,
\begin{eqnarray*}
\frac{|x-y|^{\alpha-n}}{|x|^{\alpha-n}}u^{\theta}(y)\le(1+|y|^{\alpha-n})u^{\theta}(y),
\end{eqnarray*}
and for $|y|\ge1$,
\begin{eqnarray*}
\frac{|x-y|^{\alpha-n}}{|y|^{\alpha-n}}v^{\kappa}(x)\le(1+|x|^{\alpha-n})v^{\kappa}(x).
\end{eqnarray*}
Combining these  with  $(i)$ and  using the dominated convergence theorem we have $(ii):$
\begin{eqnarray*}
a&=&\lim_{|y|\to\infty}|y|^{n-\alpha}u(y)=\lim_{|y|\to\infty}
\int_{\mathbb{R}^n}\frac{|x-y|^{\alpha-n}}{|y|^{\alpha-n}}v^{\kappa}(x)dx=\int_{\mathbb{R}^n}v^{\kappa}(x)dx<\infty,\\ and& &\\
 b&=&\lim_{|x|\to\infty}|x|^{n-\alpha}v(x)=\lim_{|x|\to\infty}
\int_{\mathbb{R}^n}\frac{|x-y|^{\alpha-n}}{|x|^{\alpha-n}}u^{\theta}(y)dy=\int_{\mathbb{R}^n}u^{\theta}(y)dy<\infty.
\end{eqnarray*}
Combining $(i)$ and $(ii)$ with \eqref{sys-1}, we have the right side inequality in $(iii)$.
\end{proof}

\begin{lm}\label{neg-lm-2}
For  $1\le n<\alpha $ and $\theta,\kappa<0$, if $(u,v)$ is a pair of positive Lebesgue measurable solutions to \eqref{sys-1}, then $u,v\in C^\infty( \mathbb{R}^n)$.
\end{lm}
\begin{proof} Again, we adopt the proof given in Li \cite{Li2004}. For $R>0,$ we can split $u$ into  following two parts
\begin{eqnarray*}
u(y)&=& \int_{|y|\le2R} |x-y|^{\alpha-n} v^{\kappa}(x)dx
+\int_{|y|>2R}|x-y|^{\alpha-n} v^{\kappa}(x)dx\\
&=&J_1(x)+J_2(x).
\end{eqnarray*}
 From  Lemma \ref{Reg-lm-1} $(i)$ we know that $J_2(x)$  can be differentiated under the integral for $|y|<R,$ so $J_2\in C^\infty({B_R})$. On the other hand, by Lemma \ref{Reg-lm-1} $(iii)$, we have $v^{\kappa}\in L^\infty(B_{2R})$, it is obvious that $J_1$ is at least H\"{o}lder continuous in ${B_R}$. Since $R>0$ is arbitrary, $u$ is H\"{o}lder continuous in ${\mathbb{R}^n}$. Thus, $u^{\theta}$ is H\"{o}lder continuous in $B_{2R}$. Similarly, we have $v,v^{\kappa}$ are H\"{o}lder continuous in $B_{2R}$. By bootstrap, we conclude that $u,v\in C^\infty(\mathbb{R}^n)$.
\end{proof}

%---------------------------------------------------------------------------------------------------------------------------------------
\subsubsection{Classification of solutions to \eqref{sys-1}}\label{subsection 2.5}
%\label{subsection 4.3}

In this part, we classify all nonnegative, non-infinity solutions to  integral system \eqref{sys-1}
 for $\theta= \kappa=(n+\alpha)/(n-\alpha)$ (that is: for $1\le n<\alpha$, $p=2n/(n+\alpha)$, $q=2n/(n-\alpha)$ in \eqref{Euler-equ-2}).

From the above discussion, we know that if  $(u, v)$ is a pair of positive measurable solutions to system \eqref{sys-1} which is not identical infinity,  then $u,\, v\in C^\infty ( \mathbb{R}^{n} ).$

\begin{thm}\label{sys-critical}
For $1\le n<\alpha$ and $\theta= \kappa=(n+\alpha)/(n-\alpha)$,  if $(u,v)$ is a pair of positive finite smooth solutions to system \eqref{sys-1},
 then $u,v$ must be the following forms on $ \mathbb{R}^n $:
\begin{eqnarray*}
u(\xi)=c_1\big(\frac{1}{|\xi-\xi_0|^2+d^2}\big)^\frac{n-\alpha}2,\\
v(\xi)=c_2\big(\frac{1}{|\xi-\xi_0|^2+d^2}\big)^\frac{n-\alpha}2,
\end{eqnarray*} where $c_1,c_2>0,d>0, \xi_0\in\mathbb{R}^n.$

\end{thm}

\begin{rem}
If one can prove that $u$ is proportional to $v$ first, then system \eqref{sys-1} can be reduced to a single equation, and the classification result for the single equation was early obtained by Li \cite{Li2004}. However, it is not obvious to us that $u$ is proportional to $v$, even though one can show that it is the case for the classic HLS inequality. In the meantime, our current work certainly gives an answer to Li's open question 1 in \cite{Li2004}.
\end{rem}

The above theorem will be proved via the method of moving sphere, following the proof for a single equation given in Li \cite{Li2004}.

%Though, we suspect that there is a similar proof to that in our early work \cite{DZ2012-2},
%based on current reversed HLS inequality, we are unable to do it yet. Thus
%Our current proof follows a previous one given in Li \cite{Li2004}.

For $x\in \mathbb{R}^n $ and $\lambda>0,$ we define the following transform:
\[
\omega_{x,\lambda}(\xi)=\big(\frac\lambda{|\xi-x|}\big)^{n-\alpha}\omega(\xi^{x,\lambda}),\quad \forall\xi\in
 {\mathbb{R}^n }\,\backslash\{x\},
\]
where
$$\xi^{x,\lambda}=x+\frac{\lambda^2(\xi-x)}{|\xi-x|^2}$$
is the Kelvin transformation of $\xi$ with respect to $B_\lambda (x)$. Also we write $\omega^k_{x,\lambda}(\xi):=
(\omega_{x,\lambda}(\xi))^k$ for any give power $k$.
 %and $\Sigma_{x,\lambda}=\mathbb{R}^n \backslash\overline{B_\lambda (x)}.$

%First we have the following lemma.
\begin{lm}\label{Kelvin formula}
Let $1\le n< \alpha$ and $ \theta,\kappa<0$. If $(u,v)$ is a pair of positive solutions to system \eqref{sys-1}, then, for any $x\in
 \mathbb{R}^{n}$,
\begin{eqnarray}
u_{x,\lambda}(\xi)&=&\int_{\mathbb{R}^{n} }\frac{v_{x,\lambda}^{\kappa}(\eta)}
{|\xi-\eta|^{n-\alpha}}\big(\frac\lambda{|\eta-x|}\big)^{\tau_1}d\eta,\quad \forall \, \xi\in \mathbb{R}^{n} ,\label{K-1}\\
v_{x,\lambda}(\eta)&=&\int_{ \mathbb{R}^{n} }\frac{u_{x,\lambda}^\theta(\xi)}{|\xi-\eta|^{n-\alpha}}
\big(\frac\lambda{|\xi-x|}\big)^{\tau_2}d\xi,\quad \forall \, \eta\in\mathbb{R}^{n},\label{K-2}
\end{eqnarray} where $\tau_1= n+\alpha-\kappa(n-\alpha),\tau_2=n+\alpha-\theta(n-\alpha).$
Moreover,
\begin{eqnarray}
u_{x,\lambda}(\xi)-u(\xi)&=&\int_{\Sigma_{x,\lambda}}K(x,\lambda;\xi,\eta)
\big[v^{\kappa}(\eta)-\big(\frac{\lambda}{|\eta-x|}\big)^{\tau_1} v_{x,\lambda}^{\kappa}(\eta)\big]d\eta,~~~~~~~~~~~~~~~~~\label{K-3}\\
v_{x,\lambda}(\eta)-v(\eta)&=&\int_{\Sigma_{x,\lambda}}K(x,\lambda;\eta,\xi)
\big[u^\theta(\xi)-\big(\frac{\lambda}{|\xi-x|}\big)^{\tau_2}u_{x,\lambda}^\theta (\xi)\big]d\xi,~~~~~~~~~~~~~~~~~~~~~~~~\label{K-4}
\end{eqnarray} where
\[
K(x,\lambda;\xi,\eta)=\big(\frac\lambda{|\xi-x|}\big)^{n-\alpha}\frac1{|\xi^{x,\lambda}-\eta|^{n-\alpha}}-\frac1{|\xi-\eta|^{n-\alpha}}
,
\]and
\[
K(x,\lambda;\xi,\eta)>0,~~~~~~~~~~~~\text{for}~\forall \,\xi, \eta\in\Sigma_{x,\lambda},\lambda>0.
\]
\end{lm}
\begin{proof} The proof is similar to that of Lemma 5.3 in \cite{Li2004}. See also our early work \cite{DZ2012-2}. We skip details here.

\end{proof}

It is clear in Lemma \ref{Kelvin formula}
that $\tau_1=\tau_2=0$ if and only if $\theta=\kappa=\frac{n+\alpha}{n-\alpha}.$  From now on in this subsection, we assume that $\theta=\kappa=\frac{n+\alpha}{n-\alpha}.$

%Define
%\[\Sigma_{x,\lambda}^u=\{\xi\in\Sigma_{x,\lambda} \,|\,u_{x,\lambda}(\xi)<u(\xi)\},\quad\text{and}\quad
%\Sigma_{x,\lambda}^v=\{\eta\in\Sigma_{x,\lambda} \,|\,v_{x,\lambda}(\eta)<v (\eta)\}.
%\]
The next lemma indicates that the procedure of moving sphere can be started.
\begin{lm}\label{K-lm-1}
Assume the same conditions on $n, \alpha, \theta $ and $\kappa$ as those in Theorem \ref{sys-critical}. Then for any $x\in \mathbb{R}^{n} $,
there exists $\lambda_0(x)>0$ such that: $\forall \ 0<\lambda<\lambda_0(x),$
\begin{eqnarray*}
u_{x,\lambda}(\xi)&\ge& u(\xi),\quad   \forall \xi\in \Sigma_{x,\lambda},\\
v_{x,\lambda}(\eta)&\ge& v(\eta),\quad  \forall \eta \in \Sigma_{x,\lambda}.
\end{eqnarray*}
\end{lm}
\begin{proof} The proof is similar to that of Lemma 5.4 in \cite{Li2004}.
 For simplicity, we assume $x=0,$ and write $u_\lambda=u_{0,\lambda}.$

 Since $n<\alpha$ and $u\in C^1(\mathbb{R}^n)$ is a positive function, there exists $r_0\in(0, 1),$ such that
\[\nabla_\xi\big(|\xi|^{\frac{n-\alpha}2}u(\xi)\big)\cdot \xi<0,~~~~~~~~~~~~~~~~~~\forall\,0<|\xi|<r_0.
\]
%This implies, via mean value theorem,  that
%\[|\xi|^\frac{n-\alpha}2u(\xi)-\big(\frac1{|\xi|}\big)^\frac{n-
%\alpha}2u(\frac{\xi}{|\xi|^2})>0,
%~~~~~~~~~~~~~~~~~~\forall\,0<|\xi|<r_0.
%\]
Thus,
\begin{equation}\label{cla-1}
u_{\lambda}(\xi)>u(\xi),~~~~~~~~~~~~~~~~~~~~~~~\forall\,0<\lambda<|\xi|<r_0.
\end{equation}

Using Lemma \ref{Reg-lm-1} $(iii),$  we have
\begin{equation*}\label{cla-2}
u(\xi)\le C(r_0)|\xi|^{\alpha -n},~~~~~~~~~~~~ ~~~~~~~~~\forall ~~~~|\xi|\ge r_0.
\end{equation*}
% with wrong inequality above, you can get the following inequality?? MJ 4-22-2012
For small $\lambda_0\in(0,r_0)$ and any $0<\lambda<\lambda_0,$ by
$(iii)$ of Lemma \ref{Reg-lm-1} and \eqref{cla-1}
\begin{eqnarray*}
u_{\lambda}(\xi)=\big(\frac\lambda{|\xi|}\big)^{n-\alpha}u(\frac{\lambda^2\xi}{|\xi|^2})
\ge\big(\frac{|\xi|}{\lambda_0}\big)^{\alpha -n}\inf_{B_{r_0}}u\ge
u(\xi),~~~~~~~~~~~~~~~~~~~~~\quad|\xi|\ge r_0.
\end{eqnarray*}
Combining the above with \eqref{cla-1}, we conclude
 \[
 u_{x,\lambda}(\xi) \ge u(\xi), \ \ \ \, \quad ~~~~~\forall \xi \in~\Sigma_{x,\lambda}
 \]
with $x=0$ and $\lambda_0(x)=\lambda_0.$
In the same way,  we can prove the inequality for $v(\eta).$
%\[
%v_{x,\lambda}(\eta) \ge v(\eta),\quad  a.e~ \mbox{in} ~\Sigma_{x,\lambda}.
%\]
\end{proof}

\medskip
For a given $x \in {\mathbb{R}^n}$,
define
\[\bar{\lambda}(x)=sup\{\mu>0\,|\,u_{x,\lambda}(\xi)\ge u (\xi), \mbox{and} \ v_{x,\lambda}(\eta)\ge v
(\eta), \  \forall \lambda\in (0, \mu), \forall\, \xi,\eta\in \Sigma_{x,\lambda}\}.
\]
The next lemma shows: if the sphere stops, then we have conformal invariant properties for solutions.

\begin{lm}\label{K-lm-2}For some $x_0\in
\mathbb{R}^{n}$,
if $\bar{\lambda}(x_0)<\infty$, then
\begin{eqnarray*}
u_{x_0,\bar{\lambda}(x_0)}(\xi)&=&u(\xi),\quad \forall \xi \in ~\mathbb{R}^{n},\\
v_{x_0,\bar{\lambda}(x_0)}(\eta)&=&v(\eta),\quad \forall \eta \in ~\mathbb{R}^{n}.
\end{eqnarray*}
\end{lm}
\begin{proof}
Again, the proof is similar to that of Lemma 5.5 in \cite{Li2004}.

Without loss of generality, we assume that
$x_0=0$, and write $\bar{\lambda}=\bar{\lambda}(0), \, u_\lambda=u_{0,\lambda}, \, v_\lambda=v_{0,\lambda}, \,
\xi^\lambda=\xi^{0,\lambda}, \, \eta^\lambda=\eta^{0,\lambda}.$
By the definition of $\bar{\lambda}$,
\begin{eqnarray*}
u_{\bar{\lambda}}(\xi)\ge
u(\xi), \quad
v_{\bar{\lambda}}(\eta)\ge
v(\eta),\quad\forall\,|\xi|,|\eta|\ge\bar{\lambda}.
\end{eqnarray*}
Noting that $\tau_1=\tau_2=0$ for $\theta=\kappa=\frac{n+\alpha}{n-\alpha}.$
 Thus, using \eqref{K-3} and \eqref{K-4} with $x=0,\lambda=\bar{\lambda},$ and the positivity of the kernel,
 we know that there are following two cases:

  $(\mathbf{a})$ $u_{\bar{\lambda}}(\xi)=u(\xi)$ and  $v_{\bar{\lambda}}(\eta)=
v(\eta)$  for all $|\xi|,|\eta|\ge\bar{\lambda}$;
  or
  $(\mathbf{b})$ $u_{\bar{\lambda}}(\xi)>u(\xi)$ and $v_{\bar{\lambda}}(\eta)>v(\eta)$ for all $|\xi|,|\eta|\ge\bar{\lambda}$.

  We show  that case $(\mathbf{b})$ can not happen. More precisely, supposing that  $u_{\bar{\lambda}}(\xi)>u(\xi)$  and $v_{\bar{\lambda}}(\eta)>v(\eta)$ for all $|\xi|,|\eta|\ge\bar{\lambda}$, we will show that there is a
$\varepsilon_*>0$, such that, for any $\lambda\in (\bar \lambda, \bar{\lambda}+\varepsilon_*)$, $u_{{\lambda}}(\xi) \ge u(\xi)$ and $v_{{\lambda}}(\eta)\ge v(\eta)$  for any
$|\xi|,|\eta|>{\lambda}.$ This contradicts to the
definition of $\bar{\lambda}$. We will show this via two steps.

%{\bf I am here now. 4-22-2012}
 Step $1$. There is a $\varepsilon_1\in(0,1)$, such that for any
$\varepsilon< \varepsilon_1$, $\bar{\lambda}\le
\lambda\le\bar{\lambda}+\varepsilon$, if $|\xi|,|\eta|\ge\bar{\lambda}+1$, then
\begin{eqnarray*}
u_{\lambda}(\xi)-u(\xi)\ge \frac{\varepsilon_1}2|\xi|^{\alpha-n}\quad\text{and}\quad v_{\lambda}(\eta)-v(\eta)\ge \frac{\varepsilon_1}2|\eta|^{\alpha-n}.
\end{eqnarray*}

From Lemma \ref{Kelvin formula}, we know that $K(x,\lambda,\xi,z)>0$
$\forall \,
\xi,\eta\in \Sigma_{x,\lambda}$.
By \eqref{K-3} with $x=0,\lambda=\bar{\lambda},$  and  Fatou
Lemma, we know, for all $|\xi|\ge\bar{\lambda}$, that
\begin{eqnarray*}
& &\liminf_{|\xi|\to\infty}|\xi|^{n-\alpha}(u_{\bar{\lambda}}(\xi)-u(\xi))\\
&\ge&\int_{\Sigma_{\lambda}}\liminf_{|\xi|\to\infty}|\xi|^{n-\alpha} K(0,\bar{\lambda},\xi,\eta)[v^{\kappa}(\eta)
-v^{\kappa}_{\bar{\lambda}}(\eta)]d\eta\\
&=&\int_{\Sigma_{\lambda}}\big(\big(\frac{\bar{\lambda}}{|\eta|}\big)^{n-\alpha}-1\big)[v^{\kappa}(\eta)
-v^{\kappa}_{\bar{\lambda}}(\eta)]d\eta.
\end{eqnarray*}
Thus, using the positivity of $v_{\bar{\lambda}}-v$, we know that there exists $\varepsilon_2\in(0,1)$, such that
\begin{eqnarray*}
u_{\bar{\lambda}}(\xi)-u(\xi)\ge
\varepsilon_2|\xi|^{\alpha-n},~~~~~~~~~~~~~~~~\forall\,|\xi|\ge\bar{\lambda}+1.
\end{eqnarray*}

Due to the continuity of $u$,  there exists a
$\varepsilon_3\in(0,\varepsilon_2)$ such that for
$|\xi|\ge\bar{\lambda}+1$ and $\bar{\lambda}\le
\lambda\le\bar{\lambda}+\varepsilon_3$,
\begin{eqnarray*}
|u_\lambda(\xi)-u_{\bar{\lambda}}(\xi)|&=&|\big(\frac\lambda{|\xi|}\big)^{n-\alpha}u(\frac{\lambda^2\xi}{|\xi|^2})
-\big(\frac{\bar{\lambda}}{|\xi|}\big)^{n-\alpha}u(\frac{\bar{\lambda}^2\xi}{|\xi|^2})|\\
&\le& \frac{\varepsilon_3}{2}|\xi|^{\alpha-n}.
\end{eqnarray*} Thus
\begin{equation*}
u_\lambda(\xi)-u(\xi)=u_{\bar{\lambda}}(\xi)-u(\xi)+u_\lambda(\xi)-u_{\bar{\lambda}}(\xi)
\ge\frac{\varepsilon_2}{2}|\xi|^{\alpha-n},
\end{equation*}
 for all $|\xi|\ge\bar{\lambda}+1,\bar{\lambda}\le\lambda\le\bar{\lambda}+\varepsilon_2.$

Similarly, there exists $\varepsilon_4\in(0,\varepsilon_3)$ such that
\begin{eqnarray*}
v_{\lambda}(\eta)-v(\eta)\ge\frac{\varepsilon_4}2|\eta|^{\alpha-n},
\end{eqnarray*}  for all $|\eta|\ge\bar{\lambda}+1,\bar{\lambda}\le\lambda\le\bar{\lambda}+\varepsilon_4.$ Choosing $\varepsilon_1=\varepsilon_4$,
we complete the proof for Step $1$.

Step $2$. There is a $\varepsilon_* < \varepsilon_1$, such that for any $\varepsilon < \varepsilon_*$,
$\bar{\lambda}\le\lambda \le\bar{\lambda}+\varepsilon ,$ if $\xi,\eta \in
\mathbb{R}^n$ satisfy  $\lambda\le |\xi|,|\eta|\le\bar{\lambda}+1$, then
$u_{\lambda}(\xi)-u(\xi)\ge0$ and $v_\lambda(\eta)-v(\eta)\ge0.$

Let $\varepsilon_*\in(0,\varepsilon_1)$.%, which will be given precisely later.
 We have, for $\bar{\lambda}\le
\lambda\le\bar{\lambda}+\varepsilon_*$ and
$\lambda\le|\xi|\le\bar{\lambda}+1$,
\begin{eqnarray}\label{cla-5}
u_\lambda(\xi)-u(\xi)&=&\int_{\Sigma_\lambda}K(0,\lambda;\xi,\eta)\big(
v^\kappa(\eta)-v^\kappa_{\lambda}(\eta)\big)d\eta\nonumber\\
&\ge&\int_{\Sigma_{\lambda}\backslash\Sigma_{\bar{\lambda}+1}}K(0,\lambda;\xi,\eta)\big(
v^\kappa(\eta)-v^\kappa_{\lambda}(\eta)\big)d\eta\nonumber\\
&
&+\int_{\Sigma_{\bar{\lambda}+2}\backslash\Sigma_{ \bar{\lambda}+3}}K(0,\lambda;\xi,\eta)\big(
v^\kappa(\eta)-v^\kappa_{\lambda}(\eta)\big)d\eta\nonumber\\
&\ge&\int_{\Sigma_{\lambda}\backslash\Sigma_{\bar{\lambda}+1}}K(0,\lambda;\xi,\eta)\big(
v^\kappa_{\bar{\lambda}}(\eta)-v^\kappa_{\lambda}(\eta)\big)d\eta\nonumber\\
&
&+\int_{\Sigma_{\bar{\lambda}+2}\backslash\Sigma_{ \bar{\lambda}+3}}K(0,\lambda;\xi,\eta)\big(
v^\kappa(\eta)-v^\kappa_{\lambda}(\eta)\big)d\eta.
\end{eqnarray}
By Step $1$, there exists $\delta_1>0$ such that
\begin{eqnarray*}
v^\kappa(\eta)-v^\kappa_{\lambda}(\eta)\ge
\delta_1,~~~~~~~~~~~~~~~~\forall \eta\in\Sigma_{\bar{\lambda}+2}\backslash\Sigma_{ \bar{\lambda}+3}.
\end{eqnarray*}
 Since
 \[K(0,\lambda;\xi,\eta)=0,~~~~~~~~~~~~~~~~\forall\,|\xi|=\lambda,
 \]
 \[\nabla_\xi K(0,\lambda;\xi,\eta)|_{|\xi|=\lambda}=(\alpha-n)|\xi-\eta|^{\alpha-n-2}\big(|\eta|^2-|\xi|^2\big)>0,~~~~~~~~~~~~~~~~~~~~~~~~
 ~~~~\forall\,\eta \in\Sigma_{\bar{\lambda}+2}\backslash\Sigma_{ \bar{\lambda}+3},
 \]
and the function is smooth in the relevant region, it follows, also based on the positivity of kernel, that
\[K(0,\lambda;\xi,\eta)\ge\delta_2(|\xi|-\lambda),~~~~~~~~~\forall\,\lambda\le|\xi|\le\bar{\lambda}+1,
\forall \eta\in\Sigma_{\bar{\lambda}+2}\backslash\Sigma_{ \bar{\lambda}+3}.
\]
 where $\delta_2>0$ is some constant independent of $\varepsilon_*$. It is easy to see that for some constant $C>0$ (independent of $\varepsilon_*$), and $\bar{\lambda}\le\lambda\le\bar{\lambda}+\varepsilon_*$,
 \[
 |v^\kappa_{\bar{\lambda}}(\eta)-v^\kappa_{\lambda}(\eta)|\le C\varepsilon,~~~~~~~~~~
 ~\forall\,\eta\in\Sigma_{\bar{\lambda}+2}\backslash\Sigma_{ \bar{\lambda}+3},\lambda\le|\xi|\le\bar{\lambda}+1.
 \]
Using the mean value theorem, we have, for $\lambda\le|\xi|\le\bar{\lambda}+1$, that
\begin{eqnarray*}
& &\int_{\Sigma_{{\lambda}}\backslash\Sigma_{\bar \lambda +1}}K(0,\lambda;\xi,\eta)d\eta
 = \int_{\Sigma_{{\lambda}}\backslash\Sigma_{\bar \lambda +1}}
\big(\big(\frac{|\xi|}{\lambda}\big)^{\alpha-n}|\xi^\lambda-\eta|^{\alpha-n}-|\xi-\eta|^{\alpha-n}\big)d\eta\\
&=& \int_{\Sigma_{{\lambda}}\backslash\Sigma_{\bar \lambda +1}}
\big[(\big(\frac{|\xi|}{\lambda}\big)^{\alpha-n}-1)|\xi^\lambda-\eta|^{\alpha-n}+\big(|\xi^\lambda-\eta|^{\alpha-n}-|\xi-\eta|^{\alpha-n}\big)\big]d\eta\\
&\le& C(|\xi|-\lambda)+ \int_{\Sigma_{{\lambda}}\backslash\Sigma_{\bar \lambda +1}}
\big(|\xi^\lambda-\eta|^{\alpha-n}-|\xi-\eta|^{\alpha-n}\big)d\eta\\
&\le& C(|\xi|-\lambda)+ C|\xi^\lambda-\xi|\\
&\le&C(|\xi|-\lambda).
\end{eqnarray*} Thus, for  $\varepsilon\in (0, \varepsilon_*), \, \bar{\lambda}\le\lambda\le\bar{\lambda}+\varepsilon, \, \lambda\le|\xi|\le\bar{\lambda}+1$,  from \eqref{cla-5} it follows
\begin{eqnarray*}
u_\lambda(\xi)-u(\xi)&\ge&-C\varepsilon\int_{\Sigma_{\lambda}\backslash \Sigma_{\bar{\lambda}+1} }K(0,\lambda;\xi,\eta)d\eta
+\delta_1\delta_2(|\xi|-\lambda)\int_{\Sigma_{\bar{\lambda}+2}\backslash\Sigma_{\bar{\lambda}+3}}d\eta\\
&\ge&\big(\delta_1\delta_2\int_{\Sigma_{\bar{\lambda}+2}\backslash\Sigma_{\bar{\lambda}+3}}d\eta-C\varepsilon\big)(|\xi|-\lambda)\ge0.
\end{eqnarray*}
Along the same line, we can show
 \[ v_\lambda(\eta)-v(\eta)\ge0
 \] for $\lambda\le |\xi|,|\eta|\le\bar{\lambda}+1$.
Step $2$ is established. Hence we complete the proof for Lemma \ref{K-lm-2}.
\end{proof}

\medskip

The following two key calculus lemmas are needed for carrying out moving sphere procedure. Under a stronger assumption ($f\in
C^1(\mathbb{R}^n)$), these lemmas were early proved by Li and Zhu \cite{LZ1995} (see, also, Li and Zhang \cite{LZ2003}). The current forms, due to Li and Nirenberg, are adopted from Li \cite{Li2004}. See Frank and Lieb \cite{FL2010} for further extension to nonnegative measures.

\begin{lm}\label{Li-1} (Lemma~$5.7$~ in~\cite{Li2004}) For $n\ge1, \mu\in\mathbb{R},$
let $f$ be a function defined on $\mathbb{R}^n$ and valued in $(-\infty,+\infty)$ satisfying
\[
\big(\frac\lambda{|y-x|}\big)^\mu f\big(x+\frac{\lambda^2(y-x)}{|y-x|^2}\big)\ge f(y),\quad\forall\,
|y-x|\ge\lambda>0,x,y\in\mathbb{R}^n,
\]
then $f(x)=$ constant.
\end{lm}
\begin{lm}\label{Li-2} (Lemma~$5.8$~ in~\cite{Li2004})
For $n\ge1, \mu\in\mathbb{R},$
let $f\in C^0(\mathbb{R}^n),$  and $\mu\in\mathbb{R}.$ Suppose that for every $x\in\mathbb{R}^n$, there exists $\lambda\in \mathbb{R}$ such that
\[
\big(\frac\lambda{|y-x|}\big)^\mu f\big(x+\frac{\lambda^2(y-x)}{|y-x|^2}\big)= f(y), \ \ \quad ~ \forall
y\in\mathbb{R}^n\setminus\{x\}.
\] Then there are $a\ge0,d>0$ and $\bar{x}\in\mathbb{R}^n$, such that
\[
f(x)\equiv\pm a\big(\frac1 {d+|x-\bar{x}|^2}\big)^\frac{\mu}2.
\]
\end{lm}

We are now ready to give a proof to Theorem \ref{sys-critical}.

\noindent\textbf{Proof of Theorem \ref{sys-critical}.}

We first show that $\bar{\lambda}(x)$ is finite for some $x\in \mathbb{R}^n$. Otherwise, $\bar{\lambda}(x)=\infty $ for all $x
\in\mathbb{R}^n$, then for $\xi,\eta\in \mathbb{R}^n,$
\[u_{x,\lambda}(\xi)\ge u(\xi), ~\text{and}~v_{x,\lambda}(\eta)\ge v(\eta)\quad \forall\, |\xi-x|, \, |\eta-x|>\lambda,
\]
By Lemma \ref{Li-1},  we know
that $u=v=constant$, which can not satisfy \eqref{sys-1}.

Now, for a fixed $x\in\mathbb{R}^n$, we know from the definition of $\bar{\lambda}(x)$, that,
\[u_{x,\lambda}(\xi)\ge u(\xi),~~~~~~~~~~~~~\forall \,0<\lambda<\bar{\lambda}(x), \, \forall \, |\xi-x|\ge\lambda.\]
From Lemma \ref{Reg-lm-1} $(ii)$, we have, for any $\lambda\in (0, \, \bar{\lambda}(x)),$ that
\begin{eqnarray*}\label{NI-8}
0&<&a=\lim_{|\xi|\to\infty}|\xi|^{n-\alpha}u(\xi)\le\lim_{|\xi|\to\infty}|\xi|^{n-\alpha}u_{x,\lambda}(\xi)=\lambda^{n-\alpha}u(x).
\end{eqnarray*} This shows $\bar{\lambda}(x)<\infty$ for all $x\in\mathbb{R}^n$.
 Applying Lemma \ref{K-lm-2}, we know
 \[u_{x,\bar{\lambda}}(\xi)=u(\xi),\quad\text{and}~v_{x,\bar{\lambda}}(\eta)=v(\eta),~~~~~~~~~~~ ~~~~~~~~~~~\forall\, x,\xi,\eta\in\mathbb{R}^{n}.
 \]
 It then follows from Lemma \ref{Li-2}, that
 \begin{equation*}
 u(\xi)=c_1\big(\frac{1}{|\xi-\xi_0|^2+d^2}\big)^\frac{n-\alpha}2
\end{equation*} and
\begin{equation*}
 v(\xi)=c_2\big(\frac{1}{|\xi-\xi_0|^2+d^2}\big)^\frac{n-\alpha}2.
\end{equation*}
for some $c_1, \, c_2>0,d>0$ and $ \xi_0\in \mathbb{R}^{n}.$
%Using the sharp constant of inequality \eqref{HLS-ineq-2}, it can be verified
%\[c_1^\frac{n+\alpha}2 =c_2^{\frac{n-\alpha}2}.
%\]
\hfill$\Box$

%{\bf We are left to compute the precise sharp constant as in Lieb's paper. 4-21-2012}

%~~~~~~~~~~~~~~~~~~~~~~~~~~~~~~~~~~~~~~~~~~~~~~~~~~~~~~~ ~~~~~~~~~~~~~~~~~~~~~
\subsubsection{The Best constant $N^*(n, \alpha)$}\label{subection 2.6}

It follows from Theorem \ref{sys-critical} and direct computation that all extremal functions to inequality \eqref{HLS-ineq-2} can be represented by
$$
F(\xi)=a(1-\xi\cdot \eta)^{-\frac{n+\alpha} 2}
$$
for some $a>0$ and $\eta\in \mathbb{R}^{n+1}$ with $|\eta|<1$. In particular, $F(\xi)=1 $ is an extremal function.

%
%\begin{lm}\label{best constant}
%For $\lambda=n-\alpha$, $p=\frac {2n}{n+\alpha}$ and $\alpha>n$,
% \[N_1^*(n,\lambda)=\pi^{\lambda /2} \frac{\Gamma(n/2-\lambda/2)}{\Gamma(n-\lambda/2)} \{\frac{\Gamma(n/2)}{\Gamma(n)} \}^{-1+\lambda/n}.
% \]
%\end{lm}
%\begin{proof}
% From Theorem \ref{sys-critical}, we know that $f(x)=(1+|x|^2)^{-n/p}$ is an extremal function for inequality \eqref{HLS-ineq-1} with the best constant. It follows that
%

%We point out the integral $\int_{\mathbb{S}^n}|\xi-\eta|^{\alpha-n}d\eta$ is independent of $\xi$ using rotational invariance of the measure. Thus we can select a special point $\xi=(0,0,\cdots,0,1)$, the north polar.
Note that
%ing that $|\xi-\eta|=2(1-\xi\cdot\eta)$ on $\mathbb{S}^n$, it following from usual change of polar coordinates that
\begin{eqnarray*}
\int_{\mathbb{S}^n}|\xi-\eta|^{\alpha-n}d\eta =2^{\alpha-1}|\mathbb{S}^{n-1}|\frac{\Gamma(n/2)\Gamma(\alpha/2)}{\Gamma((n+\alpha)/2)}.
%&=&
%2^\frac{\alpha-n}{2}\int_{\mathbb{S}^n}(1-\eta_{n+1})^{\frac{\alpha-n}2}d\eta\\
%&=&2^\frac{\alpha-n}{2}|\mathbb{S}^{n-1}|\int_0^\pi(\sin\theta)^{n-1}(1-\cos\theta)^{\frac{\alpha-n}2}d\theta \\
%&=&2^{\frac{\alpha-n}{2}+1}|\mathbb{S}^{n-1}|\int_0^\frac\pi2(\sin2\tau)^{n-1}(1-\cos2\tau)^{\frac{\alpha-n}2}d\tau \\
%&=&2^{\frac{\alpha-n}{2}+1}|\mathbb{S}^{n-1}|\int_0^\frac\pi2(2\sin\tau\cos\tau)^{n-1}(2\sin^2\tau)^{\frac{\alpha-n}2}d\tau \\
%&=&2^{\frac{\alpha-n}{2}+1+n-1+\frac{\alpha-n}{2}}|\mathbb{S}^{n-1}|\int_0^\frac\pi2(\cos\tau)^{n-1}(\sin\tau)^{n-1+\alpha-n}d\tau \\
%&=&2^{\alpha-1}|\mathbb{S}^{n-1}|2\int_0^\frac\pi2(\cos\tau)^{n-1}(\sin\tau)^{\alpha-1}d\tau \\
%&=&2^{\alpha-1}|\mathbb{S}^{n-1}|\frac{\Gamma(n/2)\Gamma(\alpha/2)}{\Gamma((n+\alpha)/2)}.
\end{eqnarray*}
%Note that the fact
%\[|\mathbb{S}^n|= 2\pi^{\frac{n+1}2}\big[\Gamma\big(\frac{n+1}2\big)\big]^{-1}=2^{n}\pi^{\frac{n}2}\frac{2^{1-n}\sqrt{\pi}}{\Gamma((n+1)/2)}=2^{n}\pi^{\frac{n}2}\frac{\Gamma(n/2)}{\Gamma(n)}.
%\]
We have
\begin{eqnarray*}
N^*(n,\alpha)
&=&|\mathbb{S}^n|^{\frac1q-\frac1p}\int_{\mathbb{S}^n}|\xi-\eta|^{\alpha-n}d\eta\\
&=&|\mathbb{S}^n|^{ -\frac\alpha n}2^{\alpha-1}|\mathbb{S}^{n-1}|\frac{\Gamma(n/2)\Gamma(\alpha/2)}{\Gamma((n+\alpha)/2)}\\
&=&\pi^{\frac{n-\alpha}2}\frac{ \Gamma(\alpha/2)}{\Gamma((n+\alpha)/2)}\big(\frac{\Gamma(n/2)}{\Gamma(n)}\big)^{ -\frac\alpha n}.
\end{eqnarray*}
We hereby complete the proof of Theorem \ref{ext-HLS-exis}.

\medskip

 \vskip 1cm
\noindent {\bf Acknowledgements}\\
The main results  were announced by M. Zhu in a seminar talk at Princeton in April of 2012. M. Zhu would like to thank A. Chang and P. Yang for the arrangement. He also would like to thank R. Frank, E. H. Lieb and Y. Y. Li for some interesting discussions. In particular, we are indebted to R. Frank, who pointed out some errors in  an early version, and to C. Li, whose comment leads us to drop an extra condition for Theorem \ref{ext-HLS-exis} in an early version. This paper is written while J. Dou was visiting  Department of
Mathematics in the University of Oklahoma. He thanks the department
for its hospitality. The work of J. Dou is partially supported by the
National Natural Science Foundation of China (Grant No. 11101319) and
CSC project for visiting The University of Oklahoma.
%% bibliography--------------------------------------------------------------------
%\begin{center}

\end{document}